\documentclass[a4paper,twoside,11pt]{article}

\usepackage{a4wide, fancyhdr, amsmath, amssymb, mathtools, yfonts}
\usepackage{mathrsfs}
\usepackage{graphicx}
\usepackage{tikz}
\usepackage[all]{xy}
\usepackage[utf8]{inputenc}
\usepackage{amsthm}
\usepackage[english]{babel}
\usepackage{chngcntr}
\usepackage{ifthen}
\usepackage{calc}
\usepackage{hyperref}
\usepackage[capitalize]{cleveref}
\usepackage{authblk}
\usepackage{tikz-cd}
\usepackage{adjustbox}
\usepackage{todonotes}
\numberwithin{equation}{section}

\newcommand{\Z}{\mathbb{Z}}
\newcommand{\Q}{\mathbb{Q}}

\newcommand\Gal{\mathrm{Gal}}

\newcommand{\mC}{\mathcal{C}}

\newcommand{\Piy}{\Pi \mathbf{y}}
\newcommand{\Piz}{\Pi \mathbf{z}}
\newcommand{\PiY}{\Pi \mathbf{Y}}

\newcommand{\hy}[1]{\hat{y}_{#1}}
\newcommand{\hyz}[2]{\hat{y}_{{#1}{#2}}}
\newcommand{\hyi}{\hat{y}_i}
\newcommand{\hyij}{\hat{y}_{ij}}

\newtheorem{lemma}{Lemma}[section]
\newtheorem{theorem}[lemma]{Theorem}
\newtheorem{proposition}[lemma]{Proposition}
\newtheorem{corollary}[lemma]{Corollary}
\newtheorem{mydef}[lemma]{Definition}

\newtheorem{remark}{Remark}

\title{\vspace{-\baselineskip}\sffamily\bfseries The $4$-rank of class groups of $K(\sqrt{n})$}
 
\author[1]{Peter Koymans\thanks{Vivatsgasse 7, 53111  Bonn, Germany, koymans@mpim-bonn.mpg.de}}
\author[1]{Adam Morgan\thanks{Vivatsgasse 7, 53111  Bonn, Germany, a.j.morgan44@gmail.com}}
\author[1]{Harry Smit\thanks{Vivatsgasse 7, 53111  Bonn, Germany, smit@mpim-bonn.mpg.de}}
\affil[1]{Max Planck Institute for Mathematics, Bonn}
 
\date{\today}

\begin{document}
\maketitle

\begin{abstract}
Let $K/\Q$ be a quadratic extension. In this paper we study the $4$-rank of the class group $\text{Cl}(K(\sqrt{n}))$, where $n$ varies over squarefree rational integers. We show that for $100\%$ of squarefree $n$, the $4$-rank is given by an explicit formula involving the $2$-rank of $\text{Cl}(K)$ and the number of prime factors of $n$ which are inert in $K/\mathbb{Q}$.   
\end{abstract}

\section{Introduction}

Class groups are among the most fundamental objects in  number theory, yet they remain relatively inaccessible, with many problems concerning their behaviour still open. To further our understanding of class groups, a fruitful philosophy has been to ask if one can at least understand their behaviour \textit{on average}. To this end, in \cite{CL}  Cohen and Lenstra  gave a beautiful conjecture predicting the behaviour of class groups in the family of all imaginary quadratic number fields, which can be roughly paraphrased as saying that, when these fields are ordered by discriminant, a given abelian group appears as a class group with probability inversely proportional to the size of its automorphism group. These original heuristics have subsequently been extended in many different directions, with important instances being the works of Cohen--Martinet \cite{CM}, Gerth \cite{G}, Bartel--Lenstra \cite{BL} and Wang--Wood \cite{WW}.   

 Several important cases of these conjectures are now known. In the case of quadratic fields,  Fouvry--Kl\"uners show  in \cite{FK,FK2} that the $4$-rank of the class group behaves as predicted by the Cohen--Lenstra heuristics, modified by Gerth to take account of the systematic subgroup  afforded by genus theory. In recent breakthrough work \cite{Smith},  Smith has  extended this work significantly, proving that for imaginary quadratic fields, the whole $2$-Sylow subgroup of the class group  behaves according to the Cohen--Lenstra heuristics. Away from the $2$-part, the work of Davenport--Heilbronn \cite{DH} determines the average size of the $3$-torsion subgroup. For more  general families of fields, see the works of Alberts--Klys \cite{AK}, Bhargava--Varma \cite{BV}, Klys \cite{K}, and Koymans--Pagano \cite{KP2}.

In the present work, we consider the behaviour of $4$-ranks of class groups in certain families of biquadratic extensions. Here for a finite abelian group $A$, the $4$-rank of $A$ is defined as
\[\textup{rk}_{4}\, A = \dim_{\mathbb{F}_2} 2A/4A=\dim_{\mathbb{F}_2} 2(A[4]).\] 
Fix a quadratic extension $K/\mathbb{Q}$ with discriminant $\Delta$, and for a squarefree integer $n$, denote by $K_n$ the biquadratic field $K(\sqrt{n})$. Varying over all squarefree $n$ gives a natural family in which to study the behaviour of class groups. In the case $K=\mathbb{Q}(i)$, in recent work of Fouvry, Pagano and the first author \cite{FKP}, it was shown that for $100\%$ of positive odd squarefree $n$ (with respect to the natural ordering) one has 
\begin{equation} \label{eq:Q_i_formula}
\textup{rk}_4\, \textup{Cl}(K_n)=\omega_3(n)-1
\end{equation}
where $\omega_3(n)$ is the number of prime divisors of $n$ which are congruent to $3$ modulo $4$. Our main result extends this to arbitrary quadratic $K/\mathbb{Q}$. To avoid assuming that $K$ has class number $1$ our method of proof, which we explain below, is essentially independent of that work. In what follows, for an integer $n$  we write $\omega(n)$ for the number of distinct prime factors of $n$, and $\omega_{\textup{inert}}(n)$ for the number of distinct prime factors of $n$ which are inert in $K/\mathbb{Q}$. When $n$ is squarefree  we write $\Delta_n$ for the discriminant of $\mathbb{Q}(\sqrt{n})$ (thus $\Delta_n \in \{n,4n\}$).  

\begin{theorem} [=\Cref{thm:combined_main}]
\label{tMain}
Let $K/\mathbb{Q}$ be a quadratic extension with discriminant $\Delta$.  Then for $100\%$ of squarefree $n$ (in the sense of \eqref{eq:100percent}) we have 
\[\textup{rk}_4\,\textup{Cl}(K_n) =\omega_{\textup{inert}}(\Delta_n)+\omega(\Delta)+\dim_{\mathbb{F}_2}\textup{Cl}(K)[2]~-~\begin{cases}3~~&~~K/\mathbb{Q}\textup{ real,}\\ 2~~&~~K/\mathbb{Q}\textup{ imaginary}.\end{cases}\]
\end{theorem}
\noindent For a brief discussion of the implicit error term in \Cref{tMain}, see \Cref{error_term_discussion}. 

In the case that $K/\mathbb{Q}$ is imaginary, genus theory gives 
\begin{equation} \label{eq:genus_theory}
\dim_{\mathbb{F}_2}\textup{Cl}(K)[2]=\omega(\Delta)-1
\end{equation} and the above formula simplifies to show that, for $100\%$ of squarefree $n$, we have 
\[
\textup{rk}_4\,\textup{Cl}(K_n) = \omega_{\textup{inert}}(\Delta_n) + 2\omega(\Delta) - 3.
\]
Since a prime is inert in $\mathbb{Q}(i)/\mathbb{Q}$ if and only if it is congruent to $3$ modulo $4$, this recovers \eqref{eq:Q_i_formula}. If one varies over prime numbers $p$ instead of squarefree integers, then the $4$-rank of $K_p$ was studied by Chan--Milovic \cite[Theorem 1]{CM1} but under some more restrictive assumptions on the quadratic field $K$. We remark also that it is not the case that the formula in \Cref{tMain} simply holds for \textit{all} squarefree $n$. When $K=\mathbb{Q}(i)$, see that table following  \cite[Theorem 1.3]{FKP} for examples when $\textup{rk}_4\textup{Cl}(K_n)$ is strictly greater than $\omega_3(n)-1$.

An easy consequence of \Cref{tMain} is that, as $n$ varies, $\textup{rk}_4\,\textup{Cl}(K_n)$ is, roughly speaking, normally distributed with mean and variance $\log \log n/2$. This follows from a variant of the  Erd\H{o}s--Kac theorem \cite[Theorem 1.3]{KLO}, which shows that the same is true of the function $\omega_\textup{inert}(\Delta_n)$. More precisely, let $\Phi(z)$ be the cumulative distribution function of the normal distribution with mean $0$ and variance $1$, i.e.
\[
\Phi(z) = \frac{1}{\sqrt{2\pi}} \int_{-\infty}^z e^{-t^2/2} dt.
\]
Also define $A(X) = \log \log X/2$ and $B(X) = \sqrt{A(X)}$.  

\begin{corollary}[=\Cref{cor:erdos_kac}] \label{erdos_kac_cor}
Let $K/\mathbb{Q}$ be a quadratic extension. Then for all real numbers $z$ we have 
\[
\lim_{X \rightarrow \infty} \frac{|\{n : |n| \leq X, \ n \textup{ squarefree}, ~\textup{rk}_4\,\textup{Cl}(K_n)- A(X) < z B(X)\}|}{|\{n : |n| \leq X, \ n \textup{ squarefree}\}|} = \Phi(z).
\]
\end{corollary}

\begin{remark}
\textup{There is a known analogy between $4$-ranks of class groups and $2$-Selmer ranks of elliptic curves, apparent in particular from the works of Heath--Brown \cite{HB,HB2} and Fouvry--Kl\"uners  \cite{FK}, and extended to higher $2$-power ranks in the work of Smith \cite{Smith}. The analogous problem to ours on the elliptic curve side is considered in work of the second author and Paterson \cite{MP}, with the analogues of \Cref{tMain} and \Cref{erdos_kac_cor} being, respectively, Theorems 1.3 and 1.1 therein. For other instances of  Erd\H{o}s--Kac type distributions arising in the study of Selmer groups, see work of Klagsbrun--Lemke Oliver \cite{KLO} and Xiong--Zaharescu \cite{XZ}.}
\end{remark}
 
As a result of \Cref{erdos_kac_cor}, one sees that for any fixed $z\geq 0$, the proportion of fields $K_n$ for which $\textup{rk}_4\,\textup{Cl}(K_n)\leq z$ tends to $0$ as $n\rightarrow \infty$. This is markedly different to the behaviour of the $4$-rank of the class group of the quadratic fields $\mathbb{Q}(\sqrt{n})$ for varying $n$, a positive proportion of which have $4$-rank equal to $r$ for every integer $r\geq 0$, as follows from the aforementioned work of Fouvry-Klüners \cite{FK}.  In this respect, the behaviour of the $4$-rank of the fields $K_n$ is more closely analagous to the behaviour of the $2$-rank of the class group of quadratic number fields which, as mentioned previously, is well understood by genus theory. To explain why this is, and to describe our additional results, let us sketch the proof of \Cref{tMain}.

\subsection{Sketch of the proof of \Cref{tMain}}
To access the $4$-rank of $\textup{Cl}(K_n)$ we study the $\mathbb{F}_2$-vector space $2\cdot \textup{Cl}^\vee(K_n)[4]$, where $\textup{Cl}^\vee(K_n)[4]$ denotes the $4$-torsion in the dual of the class group. The dimension of $2\cdot \textup{Cl}^\vee(K_n)[4]$ is equal to $\textup{rk}_4\,\textup{Cl}(K_n)$. By class field theory we may  identify $2\cdot \textup{Cl}^\vee(K_n)[4]$ with the  group 
\[
2~\textup{Hom}_{\textup{e/w-ur}}(G_{K_n},\mathbb{Z}/4\mathbb{Z}) \subseteq \textup{Hom}_{\textup{e/w-ur}}(G_{K_n},\mathbb{Z}/2\mathbb{Z})
\] 
consisting of the everywhere-unramified $\mathbb{Z}/2\mathbb{Z}$-valued characters of the absolute Galois group of $K_n$ which lift to everywhere-unramified $\mathbb{Z}/4\mathbb{Z}$-valued characters. There is a natural restriction map on characters
\[\textup{res}_{K_n/K}\colon \textup{Hom}(G_K,\mathbb{Z}/2\mathbb{Z})\longrightarrow \textup{Hom}(G_{K_n},\mathbb{Z}/2\mathbb{Z}),\]
so it is natural to ask how much of $2\cdot \textup{Hom}_{\textup{e/w-ur}}(G_{K_n},\mathbb{Z}/4\mathbb{Z})$ consists of elements arising as restriction from characters of $K$. To answer this, in \Cref{sec:selmer} we define, for each $n$, a subgroup $\textup{Sel}_{\chi_n}(G_K, \Z/2\Z)$ of $\textup{Hom}(G_K,\mathbb{Z}/2\mathbb{Z})$ consisting of characters satisfying a specified set of local conditions. As indicated by the notation, this group is naturally viewed as a Selmer group, and the general framework of Selmer structures (see e.g. \cite[II.1]{Mazur_Rubin}) provides a convenient languange for studying it. 
 Our key algebraic result, \Cref{thm:main_algebraic_thm}, shows that   for `generic' squarefree $n$, restriction induces a homomorphism
\[ \textup{Sel}_{\chi_n}(G_K, \Z/2\Z) \longrightarrow 2\cdot \textup{Cl}^\vee(K_n)[4],\]
the dimension of the kernel and cokernel of which is explicit and independent of $n$. In fact, this does not require that $K/\mathbb{Q}$ be quadratic, and may be of independent interest. Furthermore, this gives a rather explicit handle on $\textup{Cl}^\vee(K_n)[4]$ as a $\Gal(K_n/K)$-module. At this point, one could attempt to prove \Cref{tMain} by working explicitly with the subgroups  
\[\textup{Sel}_{\chi_n}(G_K, \Z/2\Z) \subseteq H^1(K,\mathbb{Z}/2\mathbb{Z})\cong K^{\ast}/K^{\ast 2}.\]
However, the possible nontriviality of the class group of $K$ makes this approach unwieldy. We  avoid this by using the basic exact sequence
\begin{equation} \label{eq:cor_sequence}
H^1(G_\mathbb{Q},\mathbb{Z}/2\mathbb{Z})\stackrel{\textup{res}_{K/\mathbb{Q}}}{\longrightarrow } H^1(G_K,\mathbb{Z}/2\mathbb{Z})\stackrel{\textup{cores}_{K/\mathbb{Q}}}{\longrightarrow}H^1(G_\mathbb{Q},\mathbb{Z}/2\mathbb{Z}) 
\end{equation}
which is readily proven by Shapiro's lemma (see \Cref{sec:cores_map}). 
 The bulk of the analytic work is involved in showing that
 \begin{itemize}
 \item[(1)] the subgroup $X_n=\textup{cores}_{K/\mathbb{Q}}(\textup{Sel}_{\chi_n}(G_K, \Z/2\Z))$ of $\mathbb{Q}^{\ast}/\mathbb{Q}^{\ast 2}$ is trivial for $100\%$ of squarefree $n$,
 \item[(2)] the Selmer group $Y_n\subseteq \mathbb{Q}^{\ast}/\mathbb{Q}^{\ast 2}$ dual to $\textup{res}_{K/\mathbb{Q}}^{-1}(\textup{Sel}_{\chi_n}(G_K, \Z/2\Z))$ (see \Cref{sec:dual_Selmer}) is trivial for $100\%$ of squarefree $n$.
 \end{itemize}
 We accomplish both steps by adapting a method due to Heath--Brown \cite{HB}, later refined by Fouvry-Klüners   \cite{FK},  for computing asymptotics for certain sums of Jacobi symbols. 
 
  Having done this, we deduce the expression for $\textup{rk}_4\,\textup{Cl}(K_n)$ given in \Cref{tMain} as a consequence of a formula, due to  Greenberg and Wiles (we recall this in \Cref{sec:wiles_greenberg}), which computes the difference in dimension between a Selmer group and its dual, and which we apply to the group  $\textup{res}_{K/\mathbb{Q}}^{-1}(\textup{Sel}_{\chi_n}(G_K, \Z/2\Z))$. 
  
\begin{remark}
\textup{As a consequence of the proof \Cref{tMain} we see that, for $100\%$ of squarefree $n$, the $4$-rank of $\textup{Cl}(K_n)$ is essentially all accounted for by quadratic characters of $\mathbb{Q}$. It is a consequence of genus theory that this is also the case for the $2$-rank of the class group of imaginary quadratic fields. In fact, it is a pleasant exercise to use the sequence  \eqref{eq:cor_sequence}, along with the formula of Greenberg and Wiles, to recover the formula \eqref{eq:genus_theory} for the $2$-rank of the class group of an imaginary quadratic field. This similarity with genus theory goes some way to explaining the divergence between the distribution of $\textup{rk}_4\,\textup{Cl}(K_n)$ given in \Cref{erdos_kac_cor}, and the distributions arising in the Cohen--Lenstra heuristics.}
\end{remark}

\subsection{Layout of the paper}
In \Cref{sec:turans_trick} we record some basic analytic estimates concerning the number of prime divisors of a given rational integer satisfying certain Chebotarev conditions. 
In \Cref{sec:selmer} we prove our main algebraic results, beginning by recalling the language of Selmer structures which we express these in. 
Across  \Cref{sec:preparation of sum,sec:cuboids} we prove our main analytic result: that the groups $X_n$ and $Y_n$ defined above are trivial for $100\%$ of $n$. In \Cref{sec:generic_case} we combine the algebraic and analytic results to prove \Cref{tMain}.
 
\subsection{Conventions}
Throughout this paper we shall make use of the cohomology of profinite groups. Take a profinite group $G$ and a $G$-module $A$. We shall always endow $A$ with the discrete topology and assume that the action of $G$ on $A$ is continuous. Similarly, our cohomology groups always have to be interpreted as continuous group cohomology.

We say that an integer is squarefree if for all primes $p$ we have that $p \mid n$ implies $p^2 \nmid n$. In particular squarefree integers are allowed to be negative. We say that a property $\mathcal{P}$ is satisfied for $100\%$ of the squarefree integers $n$ if
\begin{equation}\label{eq:100percent}
\lim_{X \rightarrow \infty} \frac{|\{n \text{ squarefree} : |n| \leq X, n \text{ satisfies } \mathcal{P}\}|}{|\{n \text{ squarefree} : |n| \leq X\}|} = 1.
\end{equation}
If $L$ is a field of characteristic $0$, we write $\overline{L}$ for a choice of algebraic closure with absolute Galois group $G_L := \Gal(\overline{L}/L)$. Throughout this paper $K$ denotes a fixed number field. All implied constants may depend on this number field $K$.

We fix algebraic closures $\overline{K}$ and $\overline{K_v}$ for every place $v$ of $K$, and an embedding $i_v\colon \overline{K} \rightarrow \overline{K_v}$ for each place $v$ of $K$. To such an embedding corresponds an embedding $G_{K_v} \rightarrow G_K$ that induces a natural restriction map $H^1(G_K, \cdot) \rightarrow H^1(G_{K_v}, \cdot)$. For a nonarchimedean place $v$ of $K$ and a $G_{K_v}$-module $A$ we define the unramified classes to be
\[
H^1_{\text{ur}}(G_{K_v}, A) := \text{ker}(H^1(G_{K_v}, A) \xrightarrow{\text{res}} H^1(G_{K_v^{\text{ur}}}, A)),
\]
where $K_v^\text{ur}$ is the maximal unramified extension of $K_v$.

\subsection*{Acknowledgements}
We thank Carlo Pagano and Ross Paterson for various insightful discussions. The authors wish to thank the Max Planck Institute for Mathematics in Bonn for its financial support, great work conditions and an inspiring atmosphere.

\section{Tur\'an's trick} \label{sec:turans_trick}
We recall the following version of Mertens' theorem.

\begin{theorem}
\label{tMertens}
Let $K$ be a number field and let $L/K$ be an abelian extension. Fix $\sigma \in \textup{Gal}(L/K)$. Then we have
\[
\sum_{\substack{N_{K/\Q}(\mathfrak{p}) \leq X \\ \textup{Art}_{L/K}(\mathfrak{p}) = \sigma}} \frac{1}{N_{K/\Q}(\mathfrak{p})} = \frac{\log \log X}{[L: K]} + O(1),
\]
where $\mathfrak{p}$ is to be omitted from the sum in case $\mathfrak{p}$ ramifies in $L/K$.
\end{theorem}

\begin{proof}
This follows from the Chebotarev density theorem and partial summation.
\end{proof}

Let $L/K$ be an abelian extension and fix $H \subseteq \Gal(L/K)$. For a rational non-zero integer $n$, we define $\omega_H(n)$ to be the number of prime divisors $\mathfrak{p}$ of $K$ dividing $n$ such that $\mathfrak{p}$ has degree $1$ and $\text{Art}_{L/K}(\mathfrak{p})$ lies in $H$ (in particular, $\mathfrak{p}$ ought not ramify in $L/K$). The following result is an immediate consequence of Theorem \ref{tMertens}.

\begin{theorem}
\label{tTuran}
Let $K$ be a number field and let $L/K$ be an abelian extension. Fix $H \subseteq \Gal(L/K)$. Then
\[
\frac{1}{2X} \sum_{\substack{1 \leq |n| \leq X}} \omega_H(n) = \log \log X \frac{|H|}{[L: K]} + O(1).
\]
and
\[
\frac{1}{2X} \sum_{\substack{1 \leq |n| \leq X}} \omega_H(n)^2 = (\log \log X)^2 \frac{|H|^2}{[L: K]^2}+ O(\log \log X).
\]
\end{theorem}

\begin{proof}
For the first part, note that
\[
\sum_{\substack{1 \leq |n| \leq X}} \omega_H(n) = \sum_{\sigma \in H} \sum_{\substack{N_{K/\Q}(\mathfrak{p}) \leq X \\ \textup{Art}_{L/K}(\mathfrak{p}) = \sigma \\ \textup{deg}(\mathfrak{p}) = 1}} \sum_{\substack{1 \leq |n| \leq X \\ n \equiv 0 \bmod \mathfrak{p}}} 1.
\]
If $\mathfrak{p}$ has degree $1$, then $n \equiv 0 \bmod \mathfrak{p}$ happens for
\[
\frac{2X}{N_{K/\Q}(\mathfrak{p})} + O(1)
\]
rational integers $1 \leq |n| \leq X$. Since primes of degree $1$ form the main contribution in Theorem \ref{tMertens}, we get the first part of Theorem \ref{tTuran}. The second part is proven similarly.
\end{proof}

It is the following corollary of Theorem \ref{tTuran} that we shall use. 

\begin{corollary}
\label{cTuran}
Let $L/K$ be an abelian extension and let $A > 0$ be a real number. Then $100\%$ of the squarefree integers $n$ are such that for every $\sigma \in \Gal(L/K)$ there exist at least $A$ different primes $\mathfrak{p}$ of $K$ dividing $n$ with $\textup{Art}_{L/K}(\mathfrak{p}) = \sigma$.
\end{corollary}

\begin{proof}
Theorem \ref{tTuran} implies that
\[
\frac{1}{2X} \sum_{\substack{1 \leq |n| \leq X}} (\omega_H(n) - \log \log X)^2 = O(\log \log X),
\]
which immediately yields the corollary.
\end{proof}

\section{Selmer groups} \label{sec:selmer}
Take a number field $K$ and a discrete $G_K$-module $A$. A Selmer structure is a collection $\{\mathcal{L}_v\}_v$, where $\mathcal{L}_v$ is a subset of $H^1(G_{K_v}, A)$ for each place $v$ such that
\[
\mathcal{L}_v = H^1_{\text{ur}}(G_{K_v}, A)
\]
for all but finitely many places $v$. To a Selmer structure we can associate a Selmer group $\text{Sel}(G_K, A, \{\mathcal{L}_v\}_v) \subseteq H^1(G_K, A)$ defined by the exactness of 
\[
0 \rightarrow \text{Sel}(G_K, A, \{\mathcal{L}_v\}_v) \rightarrow H^1(G_K, A) \rightarrow \bigoplus_v H^1(G_{K_v}, A)/\mathcal{L}_v.
\]
It follows from finiteness of the class group and Dirichlet's unit theorem that $\text{Sel}(G_K, A, \{\mathcal{L}_v\}_v)$ is a finite abelian group provided that $A$ itself is finite. When the local conditions are clear, we shall drop the $\{\mathcal{L}_v\}_v$ from the notation. Write $\text{Cl}(K)$ for the class group of a number field $K$ and $\text{Cl}^\vee(K)$ for the dual class group.

\begin{lemma}
\label{lClass}
Let $K$ be a number field. Let $m \geq 1$ be an integer and equip $\Z/m\Z$ with the trivial $G_K$-action. Take $\mathcal{L}_v = H^1_{\textup{ur}}(G_{K_v}, \Z/m\Z)$ for all finite places $v$, and $\mathcal{L}_v = 0$ for the archimedean places $v$. Then
\[
\textup{Sel}(G_K, \Z/m\Z, \{\mathcal{L}_v\}_v) \cong \textup{Cl}^\vee(K)[m].
\]
\end{lemma}

\begin{proof}
Denote by $H_K$ the Hilbert class field of $K$. Class field theory yields a canonical isomorphism
\[
\textup{Cl}^\vee(K)[m] \cong H^1(\Gal(H_K/K), \Z/m\Z).
\]
Inflation gives an injective map $H^1(\Gal(H_K/K), \Z/m\Z) \rightarrow \textup{Sel}(G_K, \Z/m\Z, \{\mathcal{L}_v\}_v)$. But the local conditions force that any character $\chi \in \textup{Sel}(G_K, \Z/m\Z, \{\mathcal{L}_v\}_v)$ factors through $\Gal(H_K/K)$. This proves the lemma.
\end{proof}

\subsection{\texorpdfstring{The $4$-rank as a Selmer group}{The 4-rank as a Selmer group}}
For now we take $\mathcal{L}_v$ as in Lemma \ref{lClass} and we take $n$ to be a rational squarefree integer. Our aim is to describe the image of the corestriction map 
\[
\text{Sel}(G_{K(\sqrt{n})}, \Z/m\Z) \rightarrow \text{Sel}(G_K, \Z/m\Z)
\]
for $m \in \{2, 4\}$. Recall that the corestriction map on characters is explicitly given as follows. Given a character $\chi \in H^1(G_{K(\sqrt{n})}, \Z/m\Z)$, we define $\text{cores}(\chi)$ to be the character of $G_K$ that sends $\sigma$ to
\[
\left\{
\begin{array}{ll}
\chi(\sigma) + \chi(\tau^{-1} \sigma \tau) & \mbox{if } \sigma \in G_{K(\sqrt{n})} \\
\chi(\sigma^2) & \mbox{if } \sigma \not \in G_{K(\sqrt{n})},
\end{array}
\right.
\]
where $\tau$ is any lift of the non-trivial element of $\Gal(K(\sqrt{n})/K)$ to $G_K$. It is easy to see that this induces a map
\[
\text{Sel}(G_{K(\sqrt{n})}, \Z/m\Z) \rightarrow \text{Sel}(G_K, \Z/m\Z),
\]
which we will also call corestriction.

\begin{lemma}
\label{lCores0}
The image of the corestriction map
\[
\textup{Sel}(G_{K(\sqrt{n})}, \Z/2\Z) \rightarrow \textup{Sel}(G_K, \Z/2\Z)
\]
is zero for $100\%$ of squarefree integers $n$.
\end{lemma}

\begin{proof}
Let for now $L$ be any field of characteristic different from $2$. From Kummer theory we get an isomorphism
\[
H^1(G_L, \Z/2\Z) \cong L^\ast/L^{\ast 2},
\]
which is given by sending $\alpha \in L^\ast$ to the map
\[
\sigma \mapsto \frac{\sigma(\beta)}{\beta}
\]
where $\beta$ is any element of $\overline{L}$ satisfying $\beta^2 = \alpha$. We then have a commutative diagram
\[ 
\begin{tikzcd}
K(\sqrt{n})^\ast/K(\sqrt{n})^{\ast 2} \arrow{d}{N_{K(\sqrt{n})/K}} \arrow{r}{\cong}  & H^1(G_{K(\sqrt{n})}, \Z/2\Z) \arrow[swap]{d}{\text{cores}} \\%
K^\ast/K^{\ast 2} \arrow{r}{\cong} & H^1(G_K, \Z/2\Z),
\end{tikzcd}
\]
where $\cong$ is the Kummer isomorphism and $N_{K(\sqrt{n})/K}$ is the norm. Fix some character $\chi \in \text{Sel}(G_K, \Z/2\Z)$, which we may identify with an element $\alpha \in K^\ast/K^{\ast 2}$. Now take a character $\psi \in\text{Sel}(G_{K(\sqrt{n})}, \Z/2\Z)$, which we view as an element $\beta \in K(\sqrt{n})^\ast/K(\sqrt{n})^{\ast 2}$. The commutative diagram shows that
\[
\text{cores}(\psi) = \chi \Longleftrightarrow N_{K(\sqrt{n})/K}(\beta) = \alpha.
\]
Writing $\beta := x + y\sqrt{n}$, we view the equation $N_{K(\sqrt{n})/K}(\beta) = \alpha$ as a conic
\begin{align}
\label{eConic}
x^2 - ny^2 = \alpha z^2
\end{align}
to be solved non-trivially in $x, y, z \in K$. The solubility of equation (\ref{eConic}) implies that an odd prime $\mathfrak{p}$ of $K$ that ramifies in $K(\sqrt{n})$ must split in $K(\sqrt{\alpha})$. For a given non-trivial $\alpha$ it follows from Corollary \ref{cTuran} that this happens $0\%$ of the time. Since there are only finitely many choices for $\alpha$, the lemma follows.
\end{proof}

\begin{lemma}
\label{lCores2}
Let $K$ be a number field and let $n$ be a squarefree integer. Suppose that the image of the corestriction map
\begin{align}
\label{eCores0}
\textup{Sel}(G_{K(\sqrt{n})}, \Z/2\Z) \rightarrow \textup{Sel}(G_K, \Z/2\Z)
\end{align}
is zero. Further suppose that the odd ramified primes in $K(\sqrt{n})/K$ represent every class of the ray class group $\textup{Cl}(K, 8 \infty)$ of $K$ of conductor $8 \infty$. Then
\[
\textup{im}\left(\textup{Sel}(G_{K(\sqrt{n})}, \Z/4\Z) \xrightarrow{\textup{cores}} \textup{Sel}(G_K, \Z/4\Z)\right) = \textup{Sel}(G_K, \Z/2\Z).
\]
\end{lemma}

\begin{remark}
The first condition is satisfied for $100\%$ of squarefree integers $n$ by Lemma \ref{lCores0}. An application of Corollary \ref{cTuran} shows that the second condition is also satisfied for $100\%$ of the squarefree integers $n$.
\end{remark}

\begin{proof}
We have a commutative diagram	
\[ 
\begin{tikzcd}
\textup{Sel}(G_{K(\sqrt{n})}, \Z/4\Z) \arrow{r}{\text{cores}} \arrow[swap]{d}{\cdot 2} & \textup{Sel}(G_K, \Z/4\Z) \arrow{d}{\cdot 2} \\%
\textup{Sel}(G_{K(\sqrt{n})}, \Z/2\Z) \arrow{r}{\text{cores}}& \textup{Sel}(G_K, \Z/2\Z).
\end{tikzcd}
\]
It follows from the diagram and equation (\ref{eCores0}) that 
\[
\textup{im}\left(\textup{Sel}(G_{K(\sqrt{n})}, \Z/4\Z) \xrightarrow{\textup{cores}} \textup{Sel}(G_K, \Z/4\Z)\right) \subseteq \textup{Sel}(G_K, \Z/2\Z).
\]
It remains to prove that the other inclusion holds.
Take some non-trivial $\chi_\alpha \in \textup{Sel}(G_K, \Z/2\Z)$. Observe that $\chi_\alpha \cup \chi_\alpha$ is trivial in 
\[
H^2(G_K, \Z/2\Z) \cong \text{Br}(K)[2],
\] 
since it is locally trivial everywhere. Indeed, this follows from the fact that $K(\sqrt{\alpha})/K$ is unramified at all places. Then there exists a cyclic degree $4$ extension $L/K$ containing $K(\sqrt{\alpha})$, say $L = K(\sqrt{\alpha}, \sqrt{\beta})$. Define $L_t := K(\sqrt{\alpha}, \sqrt{t\beta})$ with $t \in K^\ast$. Then every cyclic degree $4$ extension containing $K(\sqrt{\alpha})$ is of the shape $L_t$ for some $t \in K^\ast$. If we can show that there exists a $t$ such that $L_tK(\sqrt{n})/K(\sqrt{n})$ is unramified everywhere, then we have proven the other inclusion, since $\text{cores} \circ \text{res}$ is multiplication by $2$.

For every place $v$ of $K$ and every place $w$ of $L$ above $v$, we can find $t \in K^\ast$ such that $L_{t_w}/K_v$ is unramified. Then it follows from weak approximation that we can certainly find a $t$ such that $L_t/K$ is unramified at all primes above $2$ and $\infty$. Suppose that $L_t/K$ ramifies precisely at the odd places $v_1, \dots, v_k$, corresponding to prime ideals $\mathfrak{p}_1, \dots, \mathfrak{p}_k$. We order the $v_1, \dots, v_k$ such that the places $v_1, \dots, v_\ell$ are exactly the places unramified in $K(\sqrt{n})$. By our assumption on $n$ we can find some twist $t' \in K^\ast$ and an ideal $\mathfrak{q}$ of $K$ composed entirely of odd primes ramifying in $K(\sqrt{n})$ such that
\[
t' \equiv 1 \bmod 8, \quad t' \text{ is totally positive }, \quad (t') = \mathfrak{p}_1 \cdot \ldots \cdot \mathfrak{p}_\ell \mathfrak{q}.
\]
An application of Hensel's lemma shows that $t'$ is a square in every finite extension of $\Q_2$. Replacing $L_t$ by $L_{tt'}$ we see that $L_{tt'}/K$ can ramify only at odd places that ramify in $K(\sqrt{n})$. Furthermore such places have ramification index $2$. It follows that $L_{tt'}K(\sqrt{n})/K(\sqrt{n})$ is unramified, which completes the proof of the lemma.
\end{proof}

We now define another set of local conditions, which allows us to compare the Selmer group of interest $\textup{Sel}(G_{K(\sqrt{n})}, \Z/4\Z)$ with a Selmer group over $K$ with coefficients in $\Z/2\Z$. Write $\chi_n$ for the quadratic character corresponding to $K(\sqrt{n})$. For $m$ a power of $2$ define $\Z/m\Z(\chi_n)$ to be the module $\Z/m\Z$ twisted by the character $\chi_n$, i.e. $\sigma \ast a = \chi_n(\sigma) \cdot a$, where $\chi_n(\sigma)$ is viewed in $\{\pm 1\}$. Put
\[
\mathcal{L}_{v, n} = 
\left\{
\begin{array}{ll}
\{0, \chi_n\} \subseteq H^1(G_{K_v}, \Z/m\Z(\chi_n)) & \mbox{if } v \text{ finite and ramified in } K(\sqrt{n})\\
H^1_{\text{ur}}(G_{K_v}, \Z/m\Z(\chi_n))  & \mbox{if } v \text{ finite and unramified in } K(\sqrt{n})\\
0 & \mbox{if } v \text{ archimedean}.
\end{array}
\right.
\]
We define $\textup{Sel}_{\chi_n}(G_K, \Z/m\Z(\chi_n))$ to be the Selmer group in $K$ with local conditions $\mathcal{L}_{v, n}$. When $m=2$ we omit the $\chi_n$ from the coefficients as twisting has no effect on $\mathbb{Z}/2\mathbb{Z}$. We have the following fundamental result.

\begin{theorem} 
\label{thm:main_algebraic_thm}
Assumptions as in Lemma \ref{lCores2}. Then restriction induces an exact sequence
\begin{multline*}
0 \rightarrow \{0, \chi_n\} \rightarrow \textup{Sel}_{\chi_n}(G_K, \Z/2\Z) \xrightarrow{\textup{res}} \\
\textup{ker}\left(2\textup{Sel}(G_{K(\sqrt{n})}, \Z/4\Z) \xrightarrow{\textup{lift + cores}} \textup{Sel}(G_K, \Z/2\Z)\right) \rightarrow 0,
\end{multline*}
where we view $2\textup{Sel}(G_{K(\sqrt{n})}, \Z/4\Z)$ as a subspace of $\textup{Sel}(G_{K(\sqrt{n})}, \Z/2\Z)$, and the map $\textup{lift + cores}$ is defined in equation (\ref{eDefLC}).
\end{theorem}

\begin{proof}
Take a squarefree integer $n$ satisfying the assumptions of Lemma \ref{lCores2}. By assumption the corestriction map
\begin{align}
\label{eWellDef}
\textup{Sel}(G_{K(\sqrt{n})}, \Z/2\Z) \rightarrow \textup{Sel}(G_K, \Z/2\Z)
\end{align}
is zero. Then we claim that there is a well-defined map
\begin{align}
\label{eDefLC}
2\textup{Sel}(G_{K(\sqrt{n})}, \Z/4\Z) \xrightarrow{\textup{lift + cores}} \textup{Sel}(G_K, \Z/2\Z),
\end{align}
which we describe now. Let $y \in 2\textup{Sel}(G_{K(\sqrt{n})}, \Z/4\Z)$ and take some lift $x \in \textup{Sel}(G_{K(\sqrt{n})}, \Z/4\Z)$ satisfying $y = 2x$. It follows from equation (\ref{eWellDef}) that $\text{cores}(x)$ does not depend on the choice of lift, so that the map lift and corestrict is indeed well-defined.

Now take some $\chi \in \textup{Sel}_{\chi_n}(G_K, \Z/2\Z)$. We start by verifying that
\begin{align}
\label{eClaim1}
\text{res}(\chi) \in \textup{ker}\left(2\textup{Sel}(G_{K(\sqrt{n})}, \Z/4\Z) \xrightarrow{\textup{lift + cores}} \textup{Sel}(G_K, \Z/2\Z)\right).
\end{align}
Note that $\text{res}(\chi)$ is naturally an element of $\textup{Sel}(G_{K(\sqrt{n})}, \Z/2\Z)$. The local conditions $\mathcal{L}_{v, n}$ imply that
\[
\chi \cup (\chi_n - \chi)
\]
is trivial in $H^2(G_K, \Z/2\Z)$. Now suppose that $\chi$ is not $0$ or $\chi_n$. Then $\chi \cup (\chi_n - \chi)$ gives a degree $4$ cyclic extension $L/K(\sqrt{n})$ that is dihedral over $K$. Write $F$ for the quadratic unramified extension of $K(\sqrt{n})$ given by $\chi$ and write $L = F(\sqrt{\beta})$ for some $\beta \in F^\ast$. Define for $t \in K^\ast$ the twist $L_t := F(\sqrt{t\beta})$. 

We claim that there exists some $t \in K^\ast$ such that $L_t/K(\sqrt{n})$ is unramified at all finite places. To deal with the ramification at places dividing $2$ or $\infty$, take such a place $v$ of $K$ and a place $w$ of $F$ above $v$. If $v$ ramifies in $K(\sqrt{n})$, then $\chi \cup (\chi_n - \chi)$ is already trivial in $H^2(\Gal(F_w/K_v), \mathbb{F}_2)$, so that the extension is locally split. Otherwise $\chi$ and $\chi_n - \chi$ are both unramified characters of $K_v$. This implies that there exists $t \in K^\ast$ such that $L_t/K(\sqrt{n})$ is unramified at all dyadic and archimedean places.

We now follow the argument in Lemma \ref{lCores2} to find a $t' \in K^\ast$ such that $L_{tt'}/K(\sqrt{n})$ is unramified at all dyadic and archimedean places and $L_{tt'}/K$ is unramified at all places $v$ of $K$ that are unramified in $K(\sqrt{n})$. This implies that $L_{tt'}/K(\sqrt{n})$ is unramified at all places. Indeed, otherwise inertia would be cyclic of order $4$ in the extension $L_{tt'}/K$. Note that there are only two elements of order $4$ in $\Gal(L_{tt'}/K) \cong D_4$, which correspond to the extension $L_{tt'}/K(\sqrt{n})$. This gives the desired contradiction since $F/K(\sqrt{n})$ is unramified. Hence $L_{tt'}/K(\sqrt{n})$ is unramified at all places as claimed.

We can now complete the proof of equation (\ref{eClaim1}). Indeed, let $\psi$ be an element of $\textup{Sel}(G_{K(\sqrt{n})}, \Z/4\Z)$ with fixed field $L_{tt'}$. Then clearly $2\psi = \text{res}(\chi)$. Furthermore, since $L_{tt'}/K$ is dihedral, we see that $\text{cores}(\psi) = \psi - \psi = 0$ as desired. 

Next we check that any element in the RHS of equation (\ref{eClaim1}) is in the image of restriction. Take $y \in 2\textup{Sel}(G_{K(\sqrt{n})}, \Z/4\Z)$ and take a lift $x \in \textup{Sel}(G_{K(\sqrt{n})}, \Z/4\Z)$ satisfying $y = 2x$. Write $\tau$ for the non-trivial generator of $\Gal(K(\sqrt{n})/K)$. By assumption we have that $x + \tau(x) = 0$. Now consider the corestriction map of the module $\Z/4\Z(\chi_n)$:
\[
\text{cores}_{\chi_n}\colon H^1(G_{K(\sqrt{n})}, \Z/4\Z(\chi_n)) \rightarrow H^1(G_K, \Z/4\Z(\chi_n)).
\]
Since $\Z/4\Z(\chi_n) = \Z/4\Z$ over $G_{K(\sqrt{n})}$, this induces a map
\[
\text{Sel}(G_{K(\sqrt{n})}, \Z/4\Z) \rightarrow \textup{Sel}_{\chi_n}(G_K, \Z/4\Z(\chi_n)).
\]
Since $\text{cores} = \text{cores}_{\chi_n}$ when restricted to $H^1(G_{K(\sqrt{n})}, \Z/2\Z)$, it follows that $\text{cores}_{\chi_n}$ lands in $\textup{Sel}_{\chi_n}(G_K, \Z/2\Z)$. Furthermore, since $\text{res} \circ \text{cores}_{\chi_n}$ is multiplication by $1 - \tau$, we see that
\[
\text{res}(\text{cores}_{\chi_n}(x)) = x - \tau(x) = 2x = y
\]
as desired. So far we have shown exactness of
\[
\textup{Sel}_{\chi_n}(G_K, \Z/2\Z) \rightarrow \textup{ker}\left(2\textup{Sel}(G_{K(\sqrt{n})}, \Z/4\Z) \xrightarrow{\textup{lift + cores}} \textup{Sel}(G_K, \Z/2\Z)\right) \rightarrow 0.
\]
It follows from Kummer theory that the kernel of the restriction map $H^1(G_K, \Z/2\Z) \rightarrow H^1(G_{K(\sqrt{n})}, \Z/2\Z)$ is generated by $\chi_n$. Since $\chi_n$ is clearly in $\textup{Sel}_{\chi_n}(G_K, \Z/2\Z)$, this finishes the proof.
\end{proof}

\subsection{The corestriction map} \label{sec:cores_map}
Now we restrict to $K$ quadratic over $\Q$. Write $\chi_{K/\Q}$ for the corresponding quadratic character. Since $K$ is quadratic over $\Q$, we have an exact sequence
\[
\{0, \chi_{K/\Q}\} \rightarrow H^1(G_\Q, \Z/2\Z) \xrightarrow{\text{res}} H^1(G_K, \Z/2\Z) \xrightarrow{\text{cores}} H^1(G_\Q, \Z/2\Z) \rightarrow 0.
\]
Indeed, this follows upon taking cohomology of the sequence
\[
0 \rightarrow \Z/2\Z \rightarrow \Z/2\Z[\Gal(K/\Q)] \rightarrow \Z/2\Z \rightarrow 0
\]
and applying Shapiro's lemma. For every squarefree integer $n$ this induces an exact sequence
\begin{align} \label{eq:selmer_exact_seq}
\{0, \chi_{K/\Q}\} \rightarrow \text{res}^{-1}(\textup{Sel}_{\chi_n}(G_K, \Z/2\Z)) \rightarrow \textup{Sel}_{\chi_n}(G_K, \Z/2\Z) \rightarrow \text{cores}(\textup{Sel}_{\chi_n}(G_K, \Z/2\Z)) \rightarrow 0.
\end{align}
 Note that both $\text{res}^{-1}(\textup{Sel}_{\chi_n}(G_K, \Z/2\Z))$ and $\text{cores}(\textup{Sel}_{\chi_n}(G_K, \Z/2\Z))$ are now subgroups of $H^1(G_\Q, \Z/2\Z)$. In this subsection we aim to give necessary local conditions for an element $\chi \in H^1(G_\Q, \Z/2\Z)$ to be in $\text{cores}(\textup{Sel}_{\chi_n}(G_K, \Z/2\Z))$. This will then be used in Section~\ref{sec:preparation of sum} to show that $\text{cores}(\textup{Sel}_{\chi_n}(G_K, \Z/2\Z))$ is the trivial group for $100\%$ of squarefree integers $n$.

In the previous subsection, for every squarefree integer $n$ and place $w$ of $K$, we defined subspaces $\mathcal{L}_{w, n}$ of $H^1(G_{K_w},\mathbb{Z}/2\mathbb{Z})$ given by
\[
\mathcal{L}_{w, n} = 
\left\{
\begin{array}{ll}
\{0, \chi_n\} \subseteq H^1(G_{K_w}, \Z/2\Z) & \mbox{if } w \text{ finite and ramified in } K(\sqrt{n})\\
H^1_{\text{ur}}(G_{K_w}, \Z/2\Z)  & \mbox{if } w \text{ finite and unramified in } K(\sqrt{n})\\
0 & \mbox{if } w \text{ archimedean}.
\end{array}
\right.
\]
We define the space $X_n \subseteq H^1(G_\Q, \Z/2\Z)$ by the exactness of the bottom row of the diagram

\begin{equation}
\begin{tikzcd}
\label{eCoresToQ}
0 \arrow{r} & \textup{Sel}_{\chi_n}(G_K, \Z/2\Z) \arrow{r} & H^1(G_K, \Z/2\Z) \arrow{r} \arrow{d}{\text{cores}} & \bigoplus_w H^1(G_{K_w}, \Z/2\Z)/\mathcal{L}_{w, n} \arrow{d}{\bigoplus_w \text{cores}} \\%
0 \arrow{r} & X_n \arrow{r} & H^1(G_\Q, \Z/2\Z) \arrow{r} & \bigoplus_v H^1(G_{\Q_v}, \Z/2\Z)/\text{cores}(\mathcal{L}_{w, n}).
\end{tikzcd}
\end{equation}
Note that $\text{cores}(\mathcal{L}_{w, n})$ depends only on the place $v$ of $\Q$ below $w$, so that the above commutative diagram makes sense.

\begin{lemma} \label{lem:contained_in_cor}
We have
\[
\textup{cores}(\textup{Sel}_{\chi_n}(G_K, \Z/2\Z)) \subseteq X_n.
\]
\end{lemma}

\begin{proof}
This follows from the commutative diagram in equation (\ref{eCoresToQ}).
\end{proof}

We now derive necessary local conditions for $X_n$.

\begin{theorem}
\label{tLocal}
Let $K$ be a quadratic extension of $\Q$ with discriminant $\Delta$. If $\chi_d \in X_n$ for some squarefree integer $d$, then $d$ satisfies the following conditions:
\begin{itemize}
\item if an odd prime $p$ divides $d$, then $p$ divides $\Delta$, or $p$ divides $n$ and $p$ splits in $K$; 
\item the Hilbert symbol $(d, -n)_p = 1$ for all odd primes $p$ such that $p \mid n$, $p \nmid \Delta$ and $p$ splits in $K$;
\item $d$ is a square modulo $p$ at all odd primes $p$ such that $p \mid n$, $p \nmid \Delta$ and $p$ is inert in $K$.
\end{itemize}
\end{theorem}

\begin{proof}
We start by taking an odd prime $p$ such that $p \mid n$ and $p \nmid \Delta$. Let $w$ be a place of $K$ above $p$. Then $\text{cores}(\mathcal{L}_{w, n}) = 0$ if $p$ is inert in $K$, while $\text{cores}(\mathcal{L}_{w, n}) = \{0, \chi_n\}$ if $p$ splits in $K$. This shows that $d$ must satisfy the last two conditions. Now observe that for any odd prime $p$ that does not divide $\Delta$ or $n$, we have that
\[
\text{cores}(\mathcal{L}_{w, n}) = H^1_{\text{ur}}(G_{\Q_p}, \Z/2\Z)
\]
for any place $w$ of $K$ above $p$. Therefore, if we take an odd prime $p$ dividing $d$, then this certainly implies that $p \mid \Delta n$. 

It now suffices to show that if $p \nmid \Delta$ and $p \mid n$, then $p$ splits in $K$. Suppose for the sake of contradiction that $p$ is inert in $K$. But we have already seen that $\text{cores}(\mathcal{L}_{w, n}) = 0$ for such $p$. Then $p$ clearly can not divide $d$, which gives the desired contradiction.
\end{proof}

\subsection{Dual Selmer groups} \label{sec:dual_Selmer}
We have now derived the necessary algebraic tools to show that $\textup{cores}(\textup{Sel}_{\chi_n}(G_K, \Z/2\Z))$ is trivial for $100\%$ of the squarefree integers $n$. This leaves us with computing the dimension of $\text{res}^{-1}(\textup{Sel}_{\chi_n}(G_K, \Z/2\Z))$. As we will shortly see, $\text{res}^{-1}(\textup{Sel}_{\chi_n}(G_K, \Z/2\Z))$ is itself a Selmer group. 
To compute its size, we will use a formula due to Greenberg and Wiles which relates the size of a Selmer group to that of its dual. 

Let us start by defining dual Selmer groups. Let $A$ be a finite, discrete $G_K$-module for some number field $K$. We define the dual module $A^\ast$ to be $\text{Hom}(A, \Q/\Z(1))$, where $\Q/\Z(1)$ is the Tate twist of $\Q/\Z$. Note that despite the fact that $A$ and $\Q/\Z(1)$ are both $G_K$-modules, the above $\text{Hom}$ is to be taken in the category of abelian groups. Then $A^\ast$ becomes a $G_K$-module, where the action is given by
\[
(\sigma \cdot f)(-) = \sigma \ast f(\sigma^{-1} (-)),
\]
where $\ast$ is the action of $G_K$ on $\Q/\Z(1)$. If $v$ is a place of $K$, we get the so-called local Tate pairing
\[
H^1(G_{K_v}, A) \times H^1(G_{K_v}, A^\ast) \rightarrow H^2(G_{K_v}, \Q/\Z(1)) = \text{Br}(K_v) \xrightarrow{\text{inv}_v} \Q/\Z,
\]
where the first map is induced by the cup product and the second is the local invariant map. It is known that this pairing is non-degenerate. Hence given some local conditions $\mathcal{L}_v \subseteq H^1(G_{K_v}, A)$, there is a well-defined dual condition $\mathcal{L}_v^\ast \subseteq H^1(G_{K_v}, A^\ast)$, defined by taking the orthogonal complement of $\mathcal{L}_v$ under the local Tate pairing. The dual conditions $\mathcal{L}_v^\ast$ form a Selmer structure (as a consequence of \cite[Theorem 7.2.15)]{NS}) which allows us to define the dual Selmer group as
\[
\text{Sel}(G_K, A^\ast, \{\mathcal{L}_v^\ast\}_v).
\]
In our special case we have $A = \Z/2\Z$ with trivial action, so that $A^\ast$ is also $\Z/2\Z$ with trivial action. In this case the Tate pairing is just given by the quadratic Hilbert symbol.

Now for  a place $v$ of $\mathbb{Q}$ and a place $w\mid v$ of $K$, consider the Selmer conditions 
\[
 \mathcal{L}'_{v,n}=\textup{res}_{K_w/\mathbb{Q}_v}^{-1}(\mathcal{L}_{w,n})\subseteq H^1(G_{\mathbb{Q}_v},\mathbb{Z}/2\mathbb{Z}).
\]
This subspace is independent of the choice of $w$ dividing $v$, and the Selmer group associated to the collection $\{\mathcal{L}_{v,n}'\}$ is equal to $\textup{res}^{-1}\left(\textup{Sel}_{\chi_n}(G_K,\mathbb{Z}/2\mathbb{Z})\right)$.  The dual local conditions turn out to be rather similar to the local conditions appearing in Theorem \ref{tLocal}, enabling us to show also that the dual Selmer group is trivial for $100\%$ of the squarefree integers $n$. This will make the formula of Greenberg and Wiles particularly pleasant to use. We begin by describing the local conditions $\mathcal{L}'_{v,n}$ more explicitly. 

\begin{lemma} \label{lem:explicit_L_n_prime_conditions}
Let $K/\mathbb{Q}$ be a quadratic extension with associated quadratic character $\chi_{K/\mathbb{Q}}$. Let $v$ be a place of $\mathbb{Q}$, let $w\mid v$ be any place of $K$ lying over $v$, and let $n$ be a squarefree integer.  If $v$ is archimedean then we have 
\[\mathcal{L}_{v,n}'=\begin{cases} 0~~&~~K/\mathbb{Q}\textup{ real,}\\
\textup{Hom}(G_\mathbb{R},\mathbb{Z}/2\mathbb{Z})~~&~~K/\mathbb{Q}\textup{ imaginary}. \end{cases}\]
If $v$ is nonarchimedean, writing $\Delta$ for the discriminant of $K/\mathbb{Q}$, we have 
\[\mathcal{L}_{v,n}'=\begin{cases} \{0,\chi_{K/\mathbb{Q}}\}+H^1_\textup{ur}(G_{\mathbb{Q}_v},\mathbb{Z}/2\mathbb{Z})~~&~~n\in K_w^{\ast 2},~v\mid \Delta,\\
H^1_{\textup{ur}}(G_{\mathbb{Q}_v},\mathbb{Z}/2\mathbb{Z})~~&~~n\in K_w^{\ast 2},~v\nmid \Delta,\\
\{0,\chi_n\}+\{0,\chi_{K/\mathbb{Q}}\}~~&~~n\notin K_w^{\ast 2}.
\end{cases}\]
\end{lemma}

\begin{proof}
 If $v$ is archimedean then $\mathcal{L}_{v,n}'$ is the subspace of $H^1(G_\mathbb{R},\mathbb{Z}/2\mathbb{Z})$ consisting of characters whose restriction to $G_{K_w}$ is trivial, and the claimed description follows. 
 
 From now on we suppose that $v$ is nonarchimedean. If $v$ splits in $K/\mathbb{Q}$ then the restriction map $\textup{res}_{K_w/\mathbb{Q}_v}$ is an isomorphism, and the claimed description of $\mathcal{L}_{v,n}'$ follows immediately from the definition of the the local conditions $\mathcal{L}_{w,n}$. 
 
 Next, suppose that $v$ ramifies in $K/\mathbb{Q}$, so that $K_w/\mathbb{Q}_v$ is a ramified quadratic extension. Then the subspace 
\[
\mathcal{M}_{v,n}=\begin{cases}\{0,\chi_n\}~~&~~n\notin K_w^{\ast 2}\\ H^1_\textup{ur}(G_{\mathbb{Q}_v},\mathbb{Z}/2\mathbb{Z})~~&~~ n\in K_w^{\ast 2}\end{cases}\]
of $H^1(G_{\mathbb{Q}_v},\mathbb{Z}/2\mathbb{Z})$ maps bijectively onto $\mathcal{L}_{w,n}$ under restriction, so we see that
\[
\mathcal{L}'_{v,n}=\mathcal{M}_{v,n}+\ker(\textup{res}_{K_w/\mathbb{Q}_v})=\mathcal{M}_{v,n}+\{0,\chi_{K/\mathbb{Q}}\},
\]
which has the required form. 

Finally, suppose that $v$ is inert in $K/\mathbb{Q}$ so that $K_w/\mathbb{Q}_v$ is the unique degree $2$ unramified extension of $\mathbb{Q}_v$. If $\chi_n$ is trivial when restricted to $G_{K_w}$ then $\mathcal{L}_{w,n}=H^1_{\textup{ur}}(G_{K_w},\mathbb{Z}/2\mathbb{Z})$ and we see that $\mathcal{L}_{v,n}'=H^1_{\textup{ur}}(G_{\mathbb{Q}_v},\mathbb{Z}/2\mathbb{Z})$.  Otherwise, $\chi_n$ is non-trivial when restricted to $G_{K_w}$, $\mathcal{L}_{w,n}=\{0,\chi_n\}$, and  
\[
\mathcal{L}_{v,n}'=\{0,\chi_n\}+\{0,\chi_{K/\Q}\}. 
\]
This completes the proof.
\end{proof}

Now write $Y_n \subseteq H^1(G_\Q, \Z/2\Z)$ for the dual Selmer group of $\text{res}^{-1}(\textup{Sel}_{\chi_n}(G_K, \Z/2\Z))$. The following result gives simple necessary conditions satisfied by any character in $Y_n$.

\begin{theorem}
\label{tDualLocal}
Let $K$ be a quadratic extension of $\Q$ with discriminant $\Delta$. If $\chi_d \in Y_n$ for some squarefree integer $d$, then $d$ satisfies the following conditions:
\begin{itemize}
\item if an odd prime $p$ divides $d$, then $p$ divides $\Delta$, or $p$ divides $n$ and $p$ splits in $K$; 
\item the Hilbert symbol $(d, n)_p = 1$ for all odd primes $p$ such that $p \mid n$, $p \nmid \Delta$, and $p$ splits in $K$;
\item $d$ is a square modulo $p$ at all odd primes $p$ such that $p \mid n$, $p \nmid \Delta$ and $p$ is inert in $K$.
\end{itemize}
\end{theorem}

\begin{proof}
Let us consider an odd prime $p$ such that $p \mid n$ and $p \nmid \Delta$. The claimed conditions   arise from insisting  that $\chi_d$ is orthogonal to $\mathcal{L}_{p,n}'$ at all such $p$. To see this,  note that if $p\mid n$ then  \Cref{lem:explicit_L_n_prime_conditions} shows that $\mathcal{L}_{p,n}'$ is equal to $\{0,\chi_n\}$ or $H^1(G_{\mathbb{Q}_p},\mathbb{Z}/2\mathbb{Z})$ according to whether $p$ is split or inert in $K/\mathbb{Q}$ respectively. If $p\nmid n$ then $\chi_n$ is unramified at $p$ and \Cref{lem:explicit_L_n_prime_conditions} gives $\mathcal{L}_{p,n}'=H^1_{\text{ur}}(G_{\Q_p}, \Z/2\Z)$. Computing the orthogonal complement of these subspaces gives the result. 
\end{proof}

\subsection{\texorpdfstring{A formula for the dimension of $\textup{Sel}_{\chi_n}(G_K, \Z/2\Z)$}{A formula for the dimension of Sel}} \label{sec:wiles_greenberg}
We now give a formula for the $\mathbb{F}_2$-dimension of $\textup{Sel}_{\chi_n}(G_K, \Z/2\Z)$  which is valid under certain simplifying assumptions. In the next two sections we will use the explicit descriptions of $X_n$ and $Y_n$ given above to show that these assumptions are satisfied for $100\%$ of squarefree $n$. 
 
Let us begin by stating the aforementioned formula due to Greenberg and Wiles  \cite[Proposition 1.6]{wiles1995modular} (see also \cite[Proposition 2.35]{Mazur_Rubin} or \cite[Theorem 2]{washington} for the form presented here).
 
\begin{theorem}[Greenberg, Wiles] \label{thm:wiles_greenberg}
Let $K$ be a number field, let $A$ be a finite $G_K$-module, and let $\{\mathcal{L}_v\}$ be a Selmer structure for $A$. Then we have 
\[
\frac{|\textup{Sel}(G_K,A,\{\mathcal{L}_v\}_v)|}{|\textup{Sel}(G_K, A^*,\{\mathcal{L}_v^*\}_v)|}=\frac{|A^{G_K}|}{|(A^*)^{G_K}|}\cdot \prod_{v~\textup{place of }K}\frac{|\mathcal{L}_v|}{|A^{G_{K_v}}|}.
\]
\end{theorem}
 
Now let $n$ be a squarefree integer. We  apply \Cref{thm:wiles_greenberg} to give a formula for the dimension of $\textup{res}^{-1}\left(\textup{Sel}_{\chi_n}(G_K,\mathbb{Z}/2\mathbb{Z})\right)$ assuming that the dual Selmer group $Y_n$ is trivial. In what follows, for an integer $n$ write $\omega_{\textup{inert}}(n)$ for the number of distinct prime factors of $n$ which are inert in $K/\mathbb{Q}$. When $n$ is squarefree write $\Delta_n$ for the discriminant of $\mathbb{Q}(\sqrt{n})$ (thus $\Delta_n \in \{n,4n\}$).
 
\begin{proposition} 
\label{prop:formula_for_sel_n}
Let $K$ be a quadratic extension of $\mathbb{Q}$ with discriminant $\Delta$, and let $n$ be a squarefree integer. Assume that both of the groups $X_n$ and $Y_n$ are trivial. Then we have 
\[
\textup{dim}_{\mathbb{F}_2} \textup{Sel}_{\chi_n}(G_K,\mathbb{Z}/2\mathbb{Z})=\omega(\Delta)-1+\omega_{\textup{inert}}(\Delta_n)~-~\begin{cases}1~~&~~K/\mathbb{Q}\textup{ real}, \\ 0~~&~~\textup{otherwise}.
\end{cases}
\]
\end{proposition}
 
\begin{proof}
By \Cref{lem:contained_in_cor}, the assumption that $X_n$ is trivial gives $\textup{cores}(\textup{Sel}_{\chi_n}(G_K, \Z/2\Z))=0$. Taking dimensions  in the exact sequence \eqref{eq:selmer_exact_seq} then gives
\[
\dim_{\mathbb{F}_2}\textup{Sel}_{\chi_n}(G_K,\mathbb{Z}/2\mathbb{Z})=\dim_{\mathbb{F}_2}\textup{res}^{-1}\left(\textup{Sel}_{\chi_n}(G_K,\mathbb{Z}/2\mathbb{Z})\right)-1.
\]
Since we have assumed that $Y_n$ is trivial, \Cref{thm:wiles_greenberg} then gives 
\begin{equation} \label{eq:wiles_greenberg_applied}
\dim_{\mathbb{F}_2}\textup{res}^{-1}\left(\textup{Sel}_{\chi_n}(G_K,\mathbb{Z}/2\mathbb{Z})\right)=\sum_{v~\textup{place of }\mathbb{Q}}(\dim_{\mathbb{F}_2}\mathcal{L}_{v,n}'-1).
\end{equation}
From the explicit description of the subspaces $\mathcal{L}_{v,n}'$ afforded by \Cref{lem:explicit_L_n_prime_conditions} we see that, for any place $v$ of $\mathbb{Q}$, we have 
\[
\dim_{\mathbb{F}_2}\mathcal{L}_{v,n}'=\begin{cases} 0~~&~~ v=\infty~\textup{and }  K/\mathbb{Q}~\textup{ real},\\
 2~~&~~ v\neq \infty \textup{ and }\textup{either }v\mid \Delta, ~\textup{or } v\textup{ inert in }K/\mathbb{Q}\textup{ and }v\mid \Delta_n,\\
 1~~&~~\textup{otherwise}. 
\end{cases}
\]
Substituting this into \eqref{eq:wiles_greenberg_applied} gives the result. 
\end{proof}

\section{\texorpdfstring{A distribution on the sizes of $X_n$ and $Y_n$}{A distribution on the sizes of Xn and Yn}} \label{sec:preparation of sum}
In this section we show that for $100\%$ of the squarefree $n$, the groups $X_n$ and $Y_n$ defined in the previous section consist only of the trivial character. This is a corollary of the following theorem:

\begin{theorem}
\label{thm:mainanalytic}
We have
\[
\sum_{|n| \leq X \rm{,\, sq.\, f.}} |X_n| = \frac{2}{\zeta(2)} X + O(X \log(X)^{-1/8})
\]
and the same holds when replacing $X_n$ by $Y_n$.
\end{theorem}

Both $X_n$ and $Y_n$ always contain the trivial character. This implies that 
\[
\sum_{|n| \leq X \rm{,\, sq.\, f.}} |X_n| \geq \sum_{|n| \leq X \rm{,\, sq.\, f.}} 1 = \frac{2}{\zeta(2)} X + O(\sqrt{X}),
\]
and the same holds for $Y_n$.

We begin with our ``toolbox'' of analytic theorems, after which we will show that the sum of the sizes of the $X_n$ can be bounded from above by a sum of products of Jacobi symbols (with weights and squarefree indicators). We divide the ranges of the variables of this sum into cuboids, i.e.\ products of intervals. Subsequently we show that most cuboids make a negligible contribution (i.e.\ $O(X \log (X)^{-1/8})$) to the total sum. 

\subsection{Toolbox}
Our approach rests on a triplet of theorems, each of which deals with a different type of cuboid. \cref{thm:mertens} can be used to show that cuboids with many small edges (i.e.\ intervals) make a negligible contribution. \cref{thm:siegelwalfisz} can be applied to show that cuboids with a large edge and specific other small edges are negligible, and lastly \cref{thm:doubleoscillation} deals with the cuboids with many large edges.

\begin{theorem}[Mertens] 
\label{thm:mertens}
Let $\kappa > 0$ be fixed. Uniformly for $x \geq 1$ the following bounds hold true:
\begin{align*}
&\sum_{1 \leq n \leq x} \mu^2(n)\kappa^{\omega(n)} \ll x \log(x)^{\kappa - 1} \\
&\sum_{1 \leq n \leq x} \frac{\mu^2(n)\kappa^{\omega(n)}}{n} \ll \log(x)^{\kappa}.
\end{align*}
\end{theorem}

\begin{proof}
The second inequality follows from
\[
\sum_{1 \leq n \leq x} \frac{\mu^2(n)\kappa^{\omega(n)}}{n} \leq \prod_{p \leq x} \left(1 + \frac{\kappa}{p}\right) \leq \prod_{p \leq x} \left(1 + \frac{1}{p}\right)^\kappa
\]
and Theorem \ref{tMertens} with $L = K = \Q$. The first inequality is a consequence of \cite[Theorem 2.14]{Mo--Va} and the second inequality.
\end{proof}

\begin{theorem}[Double oscillation] 
\label{thm:doubleoscillation}
Let $f(m)$ and $g(n)$ be complex squences of modulus at most one. Then for every $M, N \geq 2$ and every $\epsilon > 0$ we have
\[
\sum_{1 \leq m \leq M} \sum_{1\leq n \leq N} \mu^2(2m)\mu^2(2n) f(m)g(n) \left( \frac{m}{n} \right) \ll_\epsilon MN(M^{-1/2 + \epsilon} + N^{-1/2 + \epsilon}).
\]
\end{theorem} 

\begin{proof}
See \cite[Lemma 15]{FK} which deduces this from work of Heath--Brown \cite[Corollary 4]{HB3}.
\end{proof}

\begin{theorem}[Siegel-Walfisz] 
\label{thm:siegelwalfisz}
For any $q \geq 2$, any primitive character $\chi$ {\rm{mod}} $q$, any positive integers $k$, $r$ and any $A > 0$ we have
\[
\sum_{y \leq n \leq x} \frac{\mu^2(nr)}{k^{\omega(n)}} \chi(n) = O_{A, k}(\sqrt{q} x \log(x)^{-A} 2^{\omega(r)})
\]
for $x \geq y \geq 2$ uniformly.  
\end{theorem}

\begin{proof}
See \cite[Lemma~5.1]{FKP} for the case $k = 4$ and $\chi$ quadratic. The proof for the general case follows along the same lines.
\end{proof}

\subsection{Preparation of the sum}
Using \cref{tLocal} and \cref{tDualLocal} we can formulate an upper bound for our sum, culminating with equation \eqref{eq:yz}. Recall \cref{tLocal} and \cref{tDualLocal}:

\begin{lemma}
Any squarefree $d$ such that $\chi_d \in X_n$ resp. $\chi_d \in Y_n$ satisfies the following conditions:
\begin{enumerate}
\item if $p \mid d$ and $p \nmid 2\Delta$, then $p \mid n$ and $p$ splits in $K/\Q$;
\item if $p \mid n$, $p \nmid 2\Delta$ and $p$ splits in $K/\Q$, then $(d, -n)_p = 1$ resp. $(d, n)_p = 1$;
\item if $p \mid n$, $p \nmid 2\Delta$ and $p$ is inert in $K/\Q$, then $\left( \frac{d}{p} \right) = 1$.
\end{enumerate}
\end{lemma}

We split the second and third condition into two parts: one for the primes dividing $d$ and one for the primes not dividing $d$:

\begin{enumerate}
\item if $p \mid d$ and $p \nmid 2\Delta$, then $p \mid n$ and $p$ splits in $K/\Q$;
\item if $p \mid (d,n)$, $p \nmid 2\Delta$ and $p$ splits in $K/\Q$, then $(d, -n)_p = 1$ resp. $(d, n)_p = 1$;
\item if $p \mid n$, $p \nmid 2d\Delta$ and $p$ splits in $K/\Q$, then $(d, -n)_p = 1$ resp. $(d, n)_p = 1$;
\item if $p \mid (d,n)$, $p \nmid 2\Delta$ and $p$ is inert in $K/\Q$, then $\left( \frac{d}{p} \right) = 1$;
\item if $p \mid n$, $p \nmid 2d\Delta$ and $p$ is inert in $K/\Q$, then $\left( \frac{d}{p} \right) = 1$.
\end{enumerate}

Immediately some observations follow:
\begin{itemize}
\item The first condition implies that $d$ is a divisor of $2 \Delta n$. Moreover, when $p \mid d$ and $p \nmid 2\Delta$, we have $p \mid n$. Hence in the second and fourth condition the assumption $p \mid (d,n)$ is equivalent to $p \mid d$. 
\item If $p \mid d$ and $p \nmid 2\Delta$, then it follows from the first condition that $p$ splits in $K/\Q$. Hence the fourth condition is empty, and we can remove the assumption that $p$ splits in $K/\Q$ from the second condition.
\item If $p \mid n$ and $p \nmid d$, then the Hilbert symbols $(d, -n)_p$ and $(d, n)_p$ are equal to $\left( \frac{d}{p} \right)$, since $p$ divides $n$ exactly once. Therefore we can combine the third and fifth condition into a single condition. Note that $p$ does not ramify in $K/\Q$ as $p \nmid \Delta$. 
\item If $p \mid d$ and $p \mid n$, the Hilbert symbol $(d, -n)_p$ is equal to 
\[
\left( \frac{-1}{p} \right) \left( \frac{d/p}{p} \right)\left( \frac{-n/p}{p} \right) = \left( \frac{dn/p^2}{p} \right) = \left( \frac{d/(d,n) \cdot n/(d,n)}{p} \right),
\]
and
\[
(d, n)_p = \left( \frac{- d/(d,n) \cdot n/(d,n)}{p} \right).
\]
\end{itemize}
This results in the following.
\begin{lemma}
Any squarefree $d$ such that $\chi_d \in X_n$ resp. $\chi_d \in Y_n$ satisfies the following conditions:
\begin{enumerate}
\item $d \mid 2\Delta n$;
\item if $p \mid d$ and $p \nmid 2\Delta$, then $p$ splits in $K/\Q$;
\item if $p \mid d$ and $p \nmid 2\Delta$, then $\left( \frac{d/(d,n) \cdot n/(d,n)}{p} \right) = 1$ resp. $\left( \frac{-d/(d,n) \cdot n/(d,n)}{p} \right) = 1$;
\item if $p \mid n$ and $p \nmid 2d\Delta$, then $\left( \frac{d}{p} \right) = 1$.
\end{enumerate}
\end{lemma}

\begin{corollary}
Write $K = \Q(\sqrt{z})$ for some $z \in \Z$ squarefree. For any squarefree $n$ we have that
\[
|X_n| \leq 
\sum_{\substack{d \mid 2\Delta n \\ \rm{sq. \; f.}}} 
\prod_{\substack{p \mid d \\ p \nmid 2\Delta}} \frac12 (1 + \left( \frac{z}{p} \right)) 
\prod_{\substack{p \mid d \\ p \nmid 2\Delta}} \frac12 (1 + \left( \frac{d/(d,n) \cdot n/(d,n)}{p} \right)) 
\prod_{\substack{p \mid n \\ p \nmid 2d\Delta}} \frac12 (1 + \left( \frac{d}{p} \right)) 
\]
and
\[
|Y_n| \leq 
\sum_{\substack{d \mid 2\Delta n \\ \rm{sq. \; f.}}} 
\prod_{\substack{p \mid d \\ p \nmid 2\Delta}} \frac12 (1 + \left( \frac{z}{p} \right)) 
\prod_{\substack{p \mid d \\ p \nmid 2\Delta}} \frac12 (1 + \left(\frac{-d/(d,n) \cdot n/(d,n)}{p} \right)) 
\prod_{\substack{p \mid n \\ p \nmid 2d\Delta}} \frac12 (1 + \left( \frac{d}{p} \right)).
\]
\end{corollary}

\begin{proof}
By the previous lemma, the size of $X_n$ is bounded from above by the number of $d$ that satisfy the four conditions. By the first condition, it suffices to consider only the $d$ that are divisors of $2\Delta n$. A prime $p$ splits in $K = \Q(\sqrt{z})$ precisely when $\left( \frac{z}{p} \right) = 1$. Suppose that $p$ does not divide $2\Delta$. Then $p$ does not ramify in $K/\Q$, hence if $p$ does not split in $K/\Q$, then $\left( \frac{z}{p} \right) = -1$. This implies that $\frac12 (1 + \left( \frac{z}{p} \right))$ is an indicator function for $p$ splitting in $K/\Q$; it equals $1$ if $p$ splits and $0$ otherwise. Similarly, 
\[
\frac12 (1 + \left( \frac{d/(d,n) \cdot n/(d,n)}{p} \right))
\]
is the indicator for $\left( \frac{d/(d,n) \cdot n/(d,n)}{p} \right) = 1$ and $\frac12 (1 + \left( \frac{d}{p} \right))$ is the indicator for $\left( \frac{d}{p} \right) = 1$. A squarefree $d$ that divides $2\Delta n$ therefore meets the second, third and fourth condition precisely if 
\[
\prod_{\substack{p \mid d \\ p \nmid 2\Delta}} \frac12 (1 + \left( \frac{z}{p} \right)) 
\prod_{\substack{p \mid d \\ p \nmid 2\Delta}} \frac12 (1 + \left( \frac{d/(d,n) \cdot n/(d,n)}{p} \right)) 
\prod_{\substack{p \mid n \\ p \nmid 2d\Delta}} \frac12 (1 + \left( \frac{d}{p} \right)) = 1.
\]
The bound on the size of $Y_n$ is obtained similarly.
\end{proof}

We treat the upper bounds for $X_n$ and $Y_n$ simultaneously. Let $\pm$ denote a plus sign in the case of $X_n$ and a minus sign in the case of $Y_n$. Similarly, let $\mp$ denote a minus sign in the case of $X_n$ and a plus sign in the case of $Y_n$. We turn to evaluating 
\begin{equation}
\sum_{\substack{|n| \leq X \\ \rm{sq. \; f.}}}  \sum_{\substack{d \mid 2\Delta n \\ \rm{sq. \; f.}}} 
\prod_{\substack{p \mid d \\ p \nmid 2\Delta}} \frac12 (1 + \left( \frac{z}{p} \right)) 
\prod_{\substack{p \mid d \\ p \nmid 2\Delta}} \frac12 (1 + \left( \frac{\pm d/(d,n) \cdot n/(d,n)}{p} \right)) 
\prod_{\substack{p \mid n \\ p \nmid 2d\Delta}} \frac12 (1 + \left( \frac{d}{p} \right)).
\end{equation}
Gathering the factors $\frac{1}{2}$ and expanding the products, we get
\begin{equation}\label{eq:abc}
\sum_{\substack{|n| \leq X \\ \rm{sq. \; f.}}} \sum_{\substack{d \mid 2\Delta n \\ \rm{sq. \; f.}}}
\frac{1}{4^{\omega(d/(d, 2\Delta))}2^{\omega(n/(n,2d\Delta))}}
\sum_{\substack{a \mid d \\ (a, 2\Delta) = 1 \\ a > 0}} \left( \frac{z}{a} \right) 
\sum_{\substack{b \mid d \\ (b, 2\Delta) = 1 \\ b > 0}} \left( \frac{\pm d/(d,n) \cdot n/(d,n)}{b} \right) 
\sum_{\substack{c \mid n \\ (c, 2d\Delta) = 1 \\ c > 0}} \left( \frac{d}{c} \right).
\end{equation}

We rewrite this sum as a sum over pairwise coprime variables. With the convention that the greatest common divisor is always non-negative and the radical of a negative integer is negative, we make the following substitutions:
\begin{align*}
y_1 &= a/(a,b)\\
y_2 &= b/(a,b)\\
y_3 &= (a,b)\\
y_4 &= |d|/(d, 2ab\Delta)\\
y_5 &= c\\
y_6 &= |n|/(n, 2cd\Delta)\\
z_1 &= (d, n, 2\Delta)\\
z_2 &= d/y_1y_2y_3y_4z_1 = d/(d,n)\\
z_3 &= n/y_1y_2y_3y_4y_5y_6z_1 = {\rm{sign}}(n) \cdot (n, 2\Delta)/(d, n, 2\Delta)\\
z_4 &= {\rm{rad}}(2\Delta)/z_1z_2z_3
\end{align*}
All these variables are squarefree, pairwise coprime integers. Moreover, $y_1, \dots, y_6$ are odd and positive, $z_1$ is positive and $z_2$, $z_3$ and $z_4$ have signs equal to $d$, $n$ and $dn\Delta$ respectively. Note that
\begin{align*}
a &= y_1y_3\\
b &= y_2y_3\\
c &= y_5\\
d &= y_1y_2y_3y_4z_1z_2\\
n &= y_1y_2y_3y_4y_5y_6z_1z_3\\
(d,n) &= y_1y_2y_3y_4z_1\\
d/(d,n) &= z_2\\
n/(d,n) &= y_5y_6z_3\\
d/(d,2\Delta) &= {\rm{sign}}(z_2)y_1y_2y_3y_4\\
n/(n,2d\Delta) &= {\rm{sign}}(z_3)y_5y_6\\
{\rm{rad}}(2\Delta) &= z_1z_2z_3z_4
\end{align*}

We substitute this into \eqref{eq:abc}. For every choice of valid $n, d, a, b$ and $c$, there is exactly one choice of $y_1, \dots, y_6, z_1, \dots, z_4$. Writing $\mathbf{y} = (y_1, \dots, y_6)$, $\Piy = y_1y_2y_3y_4y_5y_6$ (and the same for $\mathbf{z}$ and $\Piz$) and $x := X/z_1|z_3|$, we obtain that \eqref{eq:abc} equals
\begin{equation}\label{eq:yz}
\sum_{\Piz = {\rm{rad}(2\Delta)}} \sum_{ \Piy \leq x}
\frac{\mu^2(\Piy \Piz)}{4^{\omega(y_1y_2y_3y_4)}2^{\omega(y_5y_6)}} \left( \frac{z}{y_1y_3} \right) \left( \frac{\pm y_5y_6z_2z_3}{y_2y_3} \right)  \left( \frac{y_1y_2y_3y_4z_1z_2}{y_5} \right).
\end{equation}

\section{Cuboids} \label{sec:cuboids}
With the notation as in the previous section, and fixing values $z_1, \dots, z_4$ (recall that $x := X/z_1|z_3|$), in order to prove \Cref{thm:mainanalytic} we now investigate the sum
\begin{equation}\label{eq:y}
\sum_{ \Piy \leq x}
\frac{\mu^2(\Piy \Piz)}{4^{\omega(y_1y_2y_3y_4)}2^{\omega(y_5y_6)}} \left( \frac{z}{y_1y_3} \right) \left( \frac{\pm y_5y_6z_2z_3}{y_2y_3} \right)  \left( \frac{y_1y_2y_3y_4z_1z_2}{y_5} \right).
\end{equation}
We will show that its value is $O(X \log(X)^{-1/8})$ provided that $d = y_1y_2y_3y_4z_1z_2 \neq 1$. 

In doing this we follow closely the argument of  Fouvry--Kl\"uners given in \cite[Section 5]{FK}, which treats sums of a similar shape to \eqref{eq:y} in order to determine the $k$-th moment of the quantity $\left|2\textup{Cl}(\mathbb{Q}(\sqrt{-n}))[4]\right|$ as $n$ varies over certain positive squarefree integers. The principal difference between that work and our analysis is that, to prove \Cref{thm:mainanalytic}, we need only compute the first moment of the quantities $|X_n|$ and $|Y_n|$, which is the analogue of the $k=1$ case of loc. cit. (higher moments would necessitate studying the result of raising  \eqref{eq:y} to the $k$-th power). This avoids the intricate study of linked indices undertaken in \cite[Section 5.6]{FK}. However, compared to the $k=1$ case of that work, the sum \eqref{eq:y} has additional complexities caused by the asymmetry between the variables $y_1,y_2,y_3,y_4$ and $y_5, y_6$ which introduces some additional case distinction into the work. 

To begin we divide the summation into cuboids: let $D:= 1 + \log(X)^{-9}$, and    we cover  $\{\Piy \leq x\}$ by 
\[
\prod_{1 \leq i \leq 6} [D^{a_i}, D^{a_{i} + 1}],
\]
where $a_i \geq 0$ and $D^{a_1 + \dots + a_6} \leq x$. We write $Y_i := D^{a_i}$ and $\PiY = Y_1Y_2Y_3Y_4Y_5Y_6$.

This covering consists of $O(\log(X)^{60})$ cuboids. The cuboids for which $\PiY \leq x$ and $D^6\PiY > x$ are not fully contained in the set $\{\Piy \leq x\}$. The following lemma shows that the part of the sum \eqref{eq:y} that is contained in these cuboids is negligible:

\begin{lemma}\label{lem:largeproduct}
We have
\[
\left|\sum_{D^{-6}x \leq \Piy \leq x}
\frac{\mu^2(\Piy \Piz)}{4^{\omega(y_1y_2y_3y_4)}2^{\omega(y_5y_6)}} \left( \frac{z}{y_1y_3} \right) \left( \frac{\pm y_5y_6z_2z_3}{y_2y_3} \right)  \left( \frac{y_1y_2y_3y_4z_1z_2}{y_5} \right)\right| = O(X \log(X)^{-1/2}).
\]
\end{lemma}

\begin{proof}
The left hand side is bounded from above by
\[
\sum_{D^{-6}x \leq \Piy \leq x}
\frac{\mu^2(\Piy)}{2^{\omega(\Piy)}} = \sum_{D^{-6}x \leq y \leq x}
\frac{\mu^2(y)}{2^{\omega(y)}} 6^{\omega(y)} = \sum_{D^{-6}x \leq y \leq x}
\mu^2(y) 3^{\omega(y)},
\]
where $6^{\omega(y)}$ counts the number of ways the prime factors of $y$ can be divided among $y_1, \dots, y_6$. By the Cauchy-Schwarz inequality it follows that 
\[
\sum_{D^{-6}x \leq y \leq x}
\mu^2(y) 3^{\omega(y)} = \sum_{1 \leq y \leq x}
\mathbf{1}_{[D^{-6}x, x]}(y) \mu^2(y) 3^{\omega(y)} \leq (x - D^{-6}x)^{1/2} \left(\sum_{1 \leq y \leq x} \mu^2(y) 9^{\omega(y)} \right)^{1/2}.
\]
Using \cref{thm:mertens} we can bound this last sum, resulting in
\begin{align*}
(x - D^{-6}x)^{1/2} \left(\sum_{1 \leq y \leq x} \mu^2(y) 9^{\omega(y)} \right)^{1/2} &\ll (x - D^{-6}x)^{1/2} (x \log(x)^8)^{1/2} \\&\ll X \log(X)^4 (1 - D^{-6})^{1/2}.
\end{align*}
It follows from $D = 1 + \log(X)^{-9}$ that $1 - D^{-6} = O(\log(X)^{-9})$, hence we find that
\[
\sum_{D^{-6}x \leq \Piy \leq x}
\frac{\mu^2(\Piy)}{2^{\omega(\Piy)}} \ll X \log(X)^{-1/2}.
\] 
\end{proof}

For the remainder of this section we assume that all cuboids we consider are contained in $\{\Piy \leq x\}$. As the previous lemma shows, the difference between the sum \eqref{eq:y} and the sum over these cuboids is $O(X \log(X)^{-1/2})$. We distinguish between different types of cuboids, which will be handled using different techniques:

\begin{mydef}
Let $\mC$ be any cuboid. We say that $Y_i$ is \emph{small} if it has value less then $M:= \log(X)^{5000}$, \emph{medium} if it is at least $M$ but less than $L := \exp(\log (X)^{1/1000})$, and \emph{large} if it is at least $L$. Moreover, we say $Y_i$ is \emph{at most medium} if it is not large, and \emph{at least medium} if it is not small. 
\end{mydef}

The idea of the proof of \cref{thm:mainanalytic} as follows: if for a cuboid $\mC$ too many $Y_i$ are at most medium, then $\mC$ does not  have enough elements to make a significant contribution, and we can use \cref{thm:mertens} (\cref{lem:smallcuboids}) to show the contribution of those cuboids is $O(X \log(X)^{-1/8})$. Then, if specific pairs of $Y_i$ are at least medium, we can show that the sums of Jacobi symbols exhibit cancellation (\cref{thm:doubleoscillation}, applied in \cref{lem:yiyjlarge}). Finally, if we have specific combinations of large and small $Y_i$, we can apply a variant of the Siegel-Walfisz theorem (\cref{thm:siegelwalfisz}, applied in Lemmas~\ref{lem:Y1Y4large}, \ref{lem:Y2Y3large}, \ref{lem:Y5Y6large}).  Compared to the work of Fouvry--Kl\"uners  mentioned above, roughly, the cuboids ruled out by \Cref{lem:largeproduct} correspond to their first family \cite[Equation (33)]{FK}, the cuboids treated in \cref{lem:smallcuboids}  correspond  to their second family \cite[Equation (37)]{FK}, the cuboids treated in \cref{lem:yiyjlarge}  correspond to their third family \cite[Equation (40)]{FK}, and the cuboids treated across Lemmas~\ref{lem:Y1Y4large}, \ref{lem:Y2Y3large}, \ref{lem:Y5Y6large}  correspond to their fourth family \cite[Equation (43)]{FK}.

\begin{lemma}\label{lem:smallcuboids}
The total contribution to equation \eqref{eq:y} by cuboids for which either of the following two conditions hold
\begin{enumerate}
\item both $Y_5$ and $Y_6$ are at most medium and at least one of $Y_1, \dots, Y_4$ is at most medium; or
\item either $Y_5$ or $Y_6$ is at most medium and at least three of $Y_1, \dots, Y_4$ are at most medium,
\end{enumerate}
is $O(X \log(X)^{-1/8})$.
\end{lemma}

\begin{proof}
Note that the absolute value of the contribution of $\mathbf{y}$ is bounded from above by 
\begin{equation}\label{eq:trivial_bound}
\frac{\mu^2(\Piy)}{4^{\omega(y_1y_2y_3y_4)}2^{\omega(y_5y_6)}}.
\end{equation}

Suppose we are in the first case. Suppose without loss of generality that $Y_4$ is at most medium.  An element $\mathbf{y}$ in one of the cuboids that meets this criterium satisfies $y_4, y_5, y_6 \leq DM$. Let $\beta$ be the product of the $y_1, y_2$ and $y_3$. By the trivial bound \eqref{eq:trivial_bound} above, the absolute value of the total contribution of all cuboids that meet the criterium is bounded from above by
\[
\sum_{y_6 \leq DM} \frac{\mu^2(y_6)}{2^{\omega(y_6)}} \sum_{y_5 \leq DM} \frac{\mu^2(y_5)}{2^{\omega(y_5)}} \sum_{y_4 \leq DM} \frac{\mu^2(y_4)}{4^{\omega(y_4)}} \sum_{\beta \leq x/y_4y_5y_6} \frac{\mu^2(\beta)}{4^{\omega(\beta)}} 3^{\omega(\beta)},
\]
where the term $3^{\omega(\beta)}$ counts the number of ways the prime factors of $\beta$ can be distributed over $y_1$, $y_2$ and $y_3$. It follows from \cref{thm:mertens} that
\[
\sum_{\beta \leq x/y_4y_5y_6} \frac{\mu^2(\beta)}{4^{\omega(\beta)}} 3^{\omega(\beta)} \ll \frac{x}{y_4y_5y_6} \log(x)^{-1/4}.
\]
The entire sum can now be bounded by more applications of \cref{thm:mertens}:
\begin{align*}
x \log(x)^{-1/4} \sum_{y_6 \leq DM} \frac{\mu^2(y_6)}{y_6 2^{\omega(y_6)}} \sum_{y_5 \leq DM} &\frac{\mu^2(y_5)}{y_5 2^{\omega(y_5)}} \sum_{y_4 \leq DM} \frac{\mu^2(y_4)}{y_4 4^{\omega(y_4)}} \\&\ll x \log(x)^{-1/4} \log(DM)^{1/2} \log(DM)^{1/2} \log(DM)^{1/4} \\&\ll X \log(X)^{-1/8}.
\end{align*}

Now suppose we are in the second case. Without loss of generality, assume that $Y_5$ is at most medium and $Y_2, Y_3$ and $Y_4$ are at most medium. Denote by $\alpha$ the product of $y_2, y_3$ and $y_4$. The contribution to the sum by cuboids that meet this criterium can be bounded from above by
\begin{equation}\label{eq:oney5y6small}
\sum_{\alpha \leq (DM)^3} \frac{\mu^2(\alpha)3^{\omega(\alpha)}}{4^{\omega(\alpha)}} \sum_{y_5 \leq DM} \frac{\mu^2(y_5)}{2^{\omega(y_5)}} \sum_{y_1y_6 \leq x/\alpha y_5} \frac{\mu^2(y_1y_6)}{4^{\omega(y_1)} 2^{\omega(y_6)}}.
\end{equation}
Let $f = (\beta \mapsto 4^{-\omega(\beta)})$ and $g = (\gamma \mapsto 2^{-\omega(\gamma)})$. Then, by \cref{thm:mertens},
\begin{align*}
\sum_{y_1y_6 \leq x/\alpha y_5} \frac{\mu^2(y_1y_6)}{4^{\omega(y_1)} 2^{\omega(y_6)}} &= \sum_{n \leq x/\alpha y_5} \mu^2(n) (f \ast g)(n) \\&= \sum_{n \leq x/\alpha y_5} \mu^2(n)\left(\frac{3}{4}\right)^{\omega(n)} \ll \frac{x}{\alpha y_5} \log(x)^{-1/4},
\end{align*}
as for any prime $p$ we have $(f \ast g)(p) = 1/2 + 1/4 = 3/4$. Similar as before, it follows from two more applications of \cref{thm:mertens} that 
\begin{align*}
x \log(x)^{-1/4} \sum_{\alpha \leq (DM)^3} \frac{\mu^2(\alpha)3^{\omega(\alpha)}}{4^{\omega(\alpha)}} \sum_{y_5 \leq DM} \frac{\mu^2(y_5)}{2^{\omega(y_5)}}  &\ll x \log(x)^{-1/4} \log((DM)^3)^{3/4} \log(DM)^{1/2} \\&\ll X \log(X)^{-1/8}.\qedhere
\end{align*}
\end{proof}

The previous lemma implies that it suffices to show that the following cuboids make a total contribution of $O(X \log(X)^{-1/8})$:
\begin{itemize}
\item cuboids for which $Y_1$, $Y_2$, $Y_3$ and $Y_4$ are large;
\item cuboids for which two of $Y_1$, $Y_2$, $Y_3$ and $Y_4$ are large and one of $Y_5$ and $Y_6$ is large;
\item cuboids for which $Y_5$ and $Y_6$ are large and $Y_1Y_2Y_3Y_4z_1z_2 \neq 1$ (so that $d \neq 1$).
\end{itemize}
The first two cases are dealt with in \cref{lem:yiyjlarge}, \ref{lem:Y1Y4large} and \ref{lem:Y2Y3large}, and the last case is dealt with in \cref{lem:Y5Y6large}, resulting in \cref{cor:all_cuboids_small}. This proves \cref{thm:mainanalytic}.

As opposed to the previous two lemmas, the following lemmas treat cuboids one at a time. We divided $\{\Piy \leq x\}$ into $O(\log(X)^{60})$ cuboids, thus any cuboid that makes a contribution of $O(X \log(X)^{-1/8 - 60})$ can be safely neglected. We make use of the following notation: we write
\[
S(\mC, \mathbf{z}) := \sum_{ \mathbf{y} \in \mC}
\frac{\mu^2(\Piy \Piz)}{4^{\omega(y_1y_2y_3y_4)}2^{\omega(y_5y_6)}} \left( \frac{z}{y_1y_3} \right) \left( \frac{\pm y_5y_6z_2z_3}{y_2y_3} \right)  \left( \frac{y_1y_2y_3y_4z_1z_2}{y_5} \right),
\]
and by $y_i \in \mC$ we mean $y_i \in [Y_i, DY_i]$. Moreover, for $1, \leq i,j \leq 6$ we write $\sum_{\hyi \in \mC}$ for a sum over $y_k \in \mC$, $k\neq i$ (i.\,e.\, a sum over five variables), and we write $\sum_{\hyij \in \mC}$ for a sum over $y_k \in \mC$, $k\neq i, j$ (i.\,e.\, a sum over four variables).

\begin{lemma} 
\label{lem:yiyjlarge}
Let $(i,j) \in \{ (1,5), (4,5), (2, 6), (3, 6)\}$ and let $\mC$ be a cuboid for which both $Y_i$ and $Y_j$ are at least medium. Then $|S(\mC, \mathbf{z})| = O(X \log(X)^{-1/8 - 60})$.
\end{lemma}

\begin{proof}
We write $f(\hat{y_i})$ for a function that does not depend on $y_i$ but might depend on $y_k$ for $1 \leq k \leq 6$, $k \neq i$. Note that we can write 
\[
\frac{\mu^2(\Piy \Piz)}{4^{\omega(y_1y_2y_3y_4)}2^{\omega(y_5y_6)}}  \left( \frac{z}{y_1y_3} \right) \left( \frac{\pm y_5y_6z_2z_3}{y_2y_3} \right)  \left( \frac{y_1y_2y_3y_4z_1z_2}{y_5} \right) = \mu^2(y_iy_j)f(\hat{y_i})g(\hat{y_j}) \left( \frac{y_i}{y_j} \right),
\]
where $f$ and $g$ are functions of modulus at most $1$. By \cref{thm:doubleoscillation} (choosing $\epsilon = 1/6$),
\[
\sum_{y_i, y_j \in \mC} \mu^2(y_iy_j) f(\hat{y_i})g(\hat{y_j}) \left( \frac{y_i}{y_j} \right) \ll D^2 Y_iY_j(Y_i^{-1/3} + Y_j^{-1/3}) \ll \frac{D^2 Y_iY_j}{\log(X)^{61}},
\]
as $Y_i, Y_j \geq M = \log(X)^{5000}$. This implies, by summing trivially over the other variables, that
\begin{align*}
|S(\mC, \mathbf{z})| \leq \sum_{\hyij \in \mC} \left| \sum_{y_i, y_j \in \mC} \mu^2(y_iy_j)f(\hat{y_i})g(\hat{y_j}) \left( \frac{y_i}{y_j} \right) \right| &\ll \left( \prod_{\substack{1 \leq k \leq 6 \\ k \neq i, j }} D Y_k \right)\frac{D^2 Y_iY_j}{\log(X)^{61}} \\&\ll \frac{X}{\log(X)^{61}},
\end{align*}
as $\PiY \leq X$. 
\end{proof}

\begin{lemma}
\label{lem:Y1Y4large}
Let $\mC$ be a cuboid in which $Y_1$ is large. Then $|S(\mC, \mathbf{z})| = O(X \log(X)^{-1/8 - 60})$.
\end{lemma}

\begin{proof}
The case where $Y_5$ is at least medium has been dealt with in \cref{lem:yiyjlarge}. We assume $Y_5$ is small. We have
\[
|S(\mC, \mathbf{z})| \leq \sum_{\hy{1} \in \mC} \frac{1}{2^{\omega(y_2y_3y_4y_5y_6)}} \left| \sum_{y_1 \in \mC} \frac{\mu^2(\Piy \Piz)}{4^{\omega(y_1)}} \left( \frac{z}{y_1} \right)\left( \frac{y_1}{y_5} \right) \right|.
\]
We show we have cancellation in the inner sum. Suppose that $y_5 \equiv 1 \bmod 4$. The inner sum then simplifies to
\[
\left| \sum_{y_1 \in \mC} \frac{\mu^2(\Piy \Piz)}{4^{\omega(y_1)}} \left( \frac{zy_5}{y_1} \right) \right|.
\]
Note that $zy_5$ is not a square: as $z \neq 1$ and $y_5 > 0$, we have $zy_5 \neq 1$. Since $z$ divides ${\rm rad}(2\Delta) = z_1z_2z_3z_4$, we see that $z$ and $y_5$ are coprime. Because $z$ and $y_5$ are also squarefree, we conclude that $zy_5$ is not a square. By \cref{thm:siegelwalfisz} we find that, for any $A>0$,
\begin{align*}
\left| \sum_{y_1 \in \mC} \frac{\mu^2(\Piy \Piz)}{4^{\omega(y_1)}} \left( \frac{z}{y_1} \right) \right| 
&= O_A(\sqrt{z}Y_1 \log(Y_1)^{-A} 2^{\omega(y_2y_3y_4y_5y_6z_1z_2z_3z_4)})
\\&= O_A(Y_1 \log(X)^{-A/1000} 2^{\omega(y_2y_3y_4y_5y_6)}),
\end{align*}
noting that $\log(Y_1) \geq \log(L) = \log(X)^{1/1000}$.

Now suppose $y_5 \equiv 3 \bmod 4$. Then the inner sum equals 
\[
\left| \sum_{y_1 \in \mC} \frac{\mu^2(\Piy \Piz)}{4^{\omega(y_1)}} \left( \frac{-zy_5}{y_1} \right) \right|.
\]
As $y_5 \equiv 3 \bmod 4$, we have $-zy_5 \neq 1$. As mentioned, $z$ and $y_5$ are coprime and squarefree, and hence $-zy_5$ is not a square. We can thus apply \cref{thm:siegelwalfisz} to show that, for any $A>0$,
\[
\left| \sum_{y_1 \in \mC} \frac{\mu^2(\Piy \Piz)}{4^{\omega(y_1)}} \left( \frac{-zy_5}{y_1} \right) \right| = O_A(Y_1 \log(X)^{2500-A/1000} 2^{\omega(y_2y_3y_4y_5y_6)}),
\]
noting that $\sqrt{y_5} \leq \sqrt{M} = \log(X)^{2500}$. In both cases we find that the inner sum is in $O_A(Y_1 \log(X)^{2500-A/1000} 2^{\omega(y_2y_3y_4y_5y_6)})$. This shows that 
\[
|S(\mC, \mathbf{z})| \leq \sum_{\hy{1} \in \mC} \frac{1}{2^{\omega(y_2y_3y_4y_5y_6)}} \left| \sum_{y_1 \in \mC} \frac{\mu^2(\Piy \Piz)}{4^{\omega(y_1)}} \left( \frac{z}{y_1} \right)\left( \frac{y_1}{y_5} \right) \right| 
\]
is in $O_A(\PiY \log(X)^{2500-A/1000}) = O_A(X \log(X)^{2500-A/1000})$, as $\PiY \leq X$. Hence, if we choose $A$ large enough (i.e.\ at least $2560625$), we obtain that $|S(\mC, \mathbf{z})| = O(X \log(X)^{-1/8 - 60})$.
\end{proof}

\begin{lemma}
\label{lem:Y2Y3large}
Let $\mC$ be a cuboid in which both $Y_2$ and $Y_3$ are large. Then $|S(\mC, \mathbf{z})| = O(X \log(X)^{-1/8 - 60})$.
\end{lemma}

\begin{proof}
The proof follows along the same lines as Lemma~\ref{lem:Y1Y4large}. The case where $Y_6$ is at least medium has been dealt with in \cref{lem:yiyjlarge}. We assume $Y_6$ is small. We have
\begin{equation}\label{eq:y2y3}
|S(\mC, \mathbf{z})| \leq \sum_{\hyz{2}{3} \in \mC} \frac{1}{2^{\omega(y_1y_4y_5y_6)}} \left| \sum_{y_2, y_3 \in \mC} \frac{\mu^2(\Piy \Piz)}{4^{\omega(y_2y_3)}} \left( \frac{z}{y_3} \right)\left( \frac{y_5}{y_2y_3} \right)\left( \frac{y_2y_3}{y_5} \right)\left( \frac{\pm y_6z_2z_3}{y_2y_3} \right)\right|.
\end{equation}

We show cancellation in the inner sum by distinguishing two cases:
\begin{itemize}
\item If $y_5 \equiv 1 \bmod 4$, the inner sum simplifies to 
\[
\sum_{y_2, y_3 \in \mC} \frac{\mu^2(\Piy \Piz)}{4^{\omega(y_2y_3)}} \left( \frac{z}{y_3} \right)\left( \frac{\pm y_6z_2z_3}{y_2y_3} \right).
\]

We bound this sum in two ways:
\begin{align}
\left| \sum_{y_2, y_3 \in \mC} \frac{\mu^2(\Piy \Piz)}{4^{\omega(y_2y_3)}} \left( \frac{z}{y_3} \right)\left( \frac{\pm y_6z_2z_3}{y_2y_3} \right) \right|
&\leq \sum_{y_2 \in \mC} \frac{1}{2^{\omega(y_2)}} \left| \sum_{y_3 \in \mC} \frac{\mu^2(\Piy \Piz)}{4^{\omega(y_3)}} \left( \frac{z}{y_3} \right)\left( \frac{\pm y_6z_2z_3}{y_3} \right)\right|  \label{eq:y2y3firstbound}
\\
\left| \sum_{y_2, y_3 \in \mC} \frac{\mu^2(\Piy \Piz)}{4^{\omega(y_2y_3)}} \left( \frac{z}{y_3} \right)\left( \frac{\pm y_6z_2z_3}{y_2y_3} \right) \right| &\leq \sum_{y_3 \in \mC} \frac{1}{2^{\omega(y_3)}} \left| \sum_{y_2 \in \mC} \frac{\mu^2(\Piy \Piz)}{4^{\omega(y_2)}} \left( \frac{\pm y_6z_2z_3}{y_2} \right) \right|. \label{eq:y2y3secondbound}
\end{align}

Suppose that $\pm y_6z_2z_3$ is a square. Then by \eqref{eq:y2y3firstbound} and \cref{thm:siegelwalfisz} we have
\begin{align*}
\left| \sum_{y_2, y_3 \in \mC} \frac{\mu^2(\Piy \Piz)}{4^{\omega(y_2y_3)}} \left( \frac{z}{y_3} \right)\left( \frac{\pm y_6z_2z_3}{y_2y_3} \right)\right| 
&\leq \sum_{y_2 \in \mC} \frac{1}{2^{\omega(y_2)}} \left| \sum_{y_3 \in \mC} \frac{\mu^2(\Piy \Piz)}{4^{\omega(y_3)}} \left( \frac{z}{y_3} \right)\right|
\\&= O_A(Y_2Y_3 \log(X)^{-A/1000} 2^{\omega(y_1y_4y_5y_6)}).
\end{align*}

If $\pm y_6z_2z_3$ is not a square, we have by \eqref{eq:y2y3secondbound} and \cref{thm:siegelwalfisz} that
\begin{align*}
\left| \sum_{y_2, y_3 \in \mC} \frac{\mu^2(\Piy \Piz)}{4^{\omega(y_2y_3)}} \left( \frac{z}{y_3} \right)\left( \frac{\pm y_6z_2z_3}{y_2y_3} \right)\right| 
&\leq \sum_{y_3 \in \mC} \frac{1}{2^{\omega(y_3)}} \left| \sum_{y_2 \in \mC} \frac{\mu^2(\Piy \Piz)}{4^{\omega(y_2)}} \left( \frac{\pm y_6z_2z_3}{y_2} \right)\right|
\\&= O_A(Y_2Y_3 \log(X)^{2500-A/1000} 2^{\omega(y_1y_4y_5y_6)}),
\end{align*}
as $\sqrt{y_6} \leq \log(X)^{2500}$.

It follows that, regardless of whether $\pm y_6z_2z_3$ is a square or not,
\[
\left| \sum_{y_2, y_3 \in \mC} \frac{\mu^2(\Piy \Piz)}{4^{\omega(y_2y_3)}} \left( \frac{z}{y_3} \right)\left( \frac{\pm y_6z_2z_3}{y_2y_3} \right)\right| = O_A(Y_2Y_3 \log(X)^{2500-A/1000} 2^{\omega(y_1y_4y_5y_6)}).
\]

\item If $y_5 \equiv 3 \bmod 4$, the inner sum simplifies to 
\[
\sum_{y_2, y_3 \in \mC} \frac{\mu^2(\Piy \Piz)}{4^{\omega(y_2y_3)}} \left( \frac{z}{y_3} \right)\left( \frac{\mp y_6z_2z_3}{y_2y_3} \right),
\]
and we can argue as before, distinguishing between $\mp y_6z_2z_3$ a square and $\mp y_6z_2z_3$ not a square.
\end{itemize}

We have shown that
\[
\left| \sum_{y_2, y_3 \in \mC} \frac{\mu^2(\Piy \Piz)}{4^{\omega(y_2y_3)}} \left( \frac{z}{y_3} \right)\left( \frac{y_5}{y_2y_3} \right)\left( \frac{y_2y_3}{y_5} \right)\left( \frac{\pm y_6z_2z_3}{y_2y_3} \right)\right|
\]
is in $O_A(Y_2Y_3 \log(X)^{2500-A/1000} 2^{\omega(y_1y_4y_5y_6)})$ for any $A>0$. It follows from \eqref{eq:y2y3} that, for any $A>0$,
\[
|S(\mC, \mathbf{z})| = O_A(X \log(X)^{2500-A/1000}).
\] 

\noindent This completes the proof.
\end{proof}

\begin{lemma}\label{lem:Y5Y6large}
Let $\mC$ be a cuboid for which $Y_5$ and $Y_6$ are large and $Y_1Y_2Y_3Y_4z_1z_2 \neq 1$. Then
\[
|S(\mC, \mathbf{z})| = O(X \log(X)^{-1/8 - 60}).
\]
\end{lemma}

\begin{proof}
By \cref{lem:yiyjlarge} we may assume that $Y_1$, $Y_2$, $Y_3$ and $Y_4$ are small. Note
\[
|S(\mC, \mathbf{z})| \leq \sum_{\hyz{5}{6} \in \mC} \frac{1}{2^{\omega(y_1y_2y_3y_4)}} \left| \sum_{y_5,y_6 \in \mC} \frac{\mu^2(\Piy \Piz)}{2^{\omega(y_5y_6)}} \left( \frac{y_1y_4z_1z_2}{y_5} \right)\left( \frac{y_2y_3}{y_5} \right)\left( \frac{y_5y_6}{y_2y_3} \right) \right|.
\]

We once again investigate the inner sum and show that it lies in $O_A(X \log(X)^{5000-A/1000})$. Suppose $y_2y_3 \neq 1$. The squarefree indicator $\mu^2(\Piy\Piz)$ implies that the inner sum equals zero if $y_2y_3$ is a square not equal to $1$, hence we may assume that $y_2y_3$ is not a square. Then an application of \cref{thm:siegelwalfisz} shows that
\begin{align*}
\left| \sum_{y_5,y_6 \in \mC} \frac{\mu^2(\Piy \Piz)}{2^{\omega(y_5y_6)}} \left( \frac{y_1y_4z_1z_2}{y_5} \right)\left( \frac{y_2y_3}{y_5} \right)\left( \frac{y_5y_6}{y_2y_3} \right) \right| 
&\leq \sum_{y_5 \in \mC} \frac{1}{2^{\omega(y_5)}} \left| \sum_{y_6 \in \mC} \frac{\mu^2(\Piy \Piz)}{2^{\omega(y_6)}}\left( \frac{y_6}{y_2y_3} \right) \right| 
\\&= O_A(Y_5Y_6 \log(X)^{5000-A/1000} 2^{\omega(y_1y_2y_3y_4)}),
\end{align*}
noting that $\sqrt{Y_2Y_3} \leq \sqrt{M^2} = \log(X)^{5000}$.

Now assume $y_2y_3 = 1$. Then $y_1y_4z_1z_2 \neq 1$ as $Y_1Y_2Y_3Y_4z_1z_2 \neq 1$. Again we may assume that $y_1y_4z_1z_2$ is not a square, as then the sum equals zero. We have
\begin{align*}
\left| \sum_{y_5,y_6 \in \mC} \frac{\mu^2(\Piy \Piz)}{2^{\omega(y_5y_6)}} \left( \frac{y_1y_4z_1z_2}{y_5} \right)\left( \frac{y_2y_3}{y_5} \right)\left( \frac{y_5y_6}{y_2y_3} \right) \right| 
&\leq \sum_{y_6 \in \mC}\frac{1}{2^{\omega(y_6)}} \left| \sum_{y_5 \in \mC} \frac{\mu^2(\Piy \Piz)}{2^{\omega(y_5)}} \left( \frac{y_1y_4z_1z_2}{y_5} \right) \right|,
\end{align*}
which, by \cref{thm:siegelwalfisz}, is also in $O_A(Y_5Y_6 \log(X)^{5000-A/1000} 2^{\omega(y_1y_2y_3y_4)})$ for any $A > 0$. 

Summing trivially over the other variables and choosing $A$ large anough, we find $|S(\mC, \mathbf{z})| = O(X \log(X)^{-1/8 - 60})$.
\end{proof}

\begin{corollary}\label{cor:all_cuboids_small}
Let $\mC$ be a cuboid for which
\begin{itemize}
\item $Y_1$, $Y_2$, $Y_3$ and $Y_4$ are large; or
\item two of $Y_1$, $Y_2$, $Y_3$ and $Y_4$ are large and one of $Y_5$ and $Y_6$ is large; or
\item $Y_5$ and $Y_6$ are large and $Y_1Y_2Y_3Y_4z_1z_2 \neq 1$.
\end{itemize}
Then $|S(\mC, z)| = O(X \log(X)^{-1/8-60})$. 
\end{corollary}

\begin{proof}
The first case is handled by either \cref{lem:Y1Y4large} or \cref{lem:Y2Y3large}. Suppose we are in the second case. Again by \cref{lem:Y1Y4large} and \cref{lem:Y2Y3large}, it suffices to show this when $Y_2$ and $Y_4$ are large. Both the case that $Y_5$ is large and the case that $Y_6$ is large falls under \cref{lem:yiyjlarge}.

Finally, the third case is precisely \cref{lem:Y5Y6large}.
\end{proof}

\begin{proof}[Proof of \cref{thm:mainanalytic}.]
As $X_n$ and $Y_n$ always contain the trivial character, we have
\[
\sum_{|n| \leq X \rm{,\, sq.\, f.}} |X_n| \geq \sum_{|n| \leq X \rm{,\, sq.\, f.}} 1 = \frac{2}{\zeta(2)} X + O(\sqrt{X}),
\]
and the same holds for $Y_n$. We established an upper bound in equation~\eqref{eq:yz}, and \cref{lem:largeproduct} show that this upper bound is approximated up to an error of $O(X \log(X)^{-1/2})$ by the following sum over cuboids:
\[
\sum_{\Piz \leq \textrm{rad}(2\Delta)} \sum_{\substack{\mC \textrm{ cuboid} \\ \PiY \leq x} } |S(\mC, \mathbf{z})|.
\]
\cref{lem:smallcuboids} and the previous corollary, noting that we have $O(\log(X)^{60})$ cuboids, then show that this equals, up to an error of $O(X \log(X)^{-1/8})$, the contribution of the cuboids with $Y_1Y_2Y_3Y_4z_1z_2 = 1$, that is, corresponding to $d = 1$. This contribution is $\frac{2}{\zeta(2)} X + O(\sqrt{X})$.

We conclude that
\[
\sum_{|n| \leq X \rm{,\, sq.\, f.}} |X_n| = \frac{2}{\zeta(2)} X + O(X \log(X)^{-1/8}),
\]
and similarly for $Y_n$.
\end{proof}

\section{\texorpdfstring{The $4$-rank in the generic case}{The 4-rank in the generic case}} 
\label{sec:generic_case}
For a quadratic extension $K/\mathbb{Q}$, we can combine the main analytic result, \Cref{thm:mainanalytic}, with the algebraic results of \Cref{sec:selmer} to give a formula for the $4$-rank of $\textup{Cl}(K(\sqrt{n}))$ which is valid for $100\%$ of squarefree $n$. As before, for an integer $n$ write $\omega_{\textup{inert}}(n)$ for the number of distinct prime factors of $n$ which are inert in $K/\mathbb{Q}$. When $n$ is squarefree, write $\Delta_n$ for the discriminant of $\mathbb{Q}(\sqrt{n})$ (thus $\Delta_n \in \{n,4n\}$). We then have the following.

\begin{theorem} 
\label{thm:combined_main}
Let $K/\mathbb{Q}$ be a quadratic extension with discriminant $\Delta$. Then for $100\%$ of squarefree $n$ we have 
\begin{equation}\label{eq:main_theorem_explicit_formula}
\textup{rk}_4\, \textup{Cl}(K(\sqrt{n}))=\omega_{\textup{inert}}(\Delta_n)+\omega(\Delta)+\dim_{\mathbb{F}_2}\textup{Cl}(K)[2]~-~\begin{cases}3~~&~~K/\mathbb{Q}\textup{ real,}\\ 2~~&~~K/\mathbb{Q}\textup{ imaginary}. \end{cases}\end{equation}
\end{theorem}

\begin{proof}
Combining \Cref{thm:main_algebraic_thm} with \Cref{lCores2} and \Cref{lCores0} we see that, for $100\%$ of squarefree $n$, we have 
\[\dim_{\mathbb{F}_2} 2\textup{Sel}(G_{K(\sqrt{n})},\mathbb{Z}/4\mathbb{Z})=\dim_{\mathbb{F}_2} \textup{Sel}_{\chi_n}(G_K,\mathbb{Z}/2\mathbb{Z})+\dim_{\mathbb{F}_2} \textup{Sel}(G_K,\mathbb{Z}/2\mathbb{Z})-1.\]
By \Cref{lClass} we can  replace $\dim_{\mathbb{F}_2}2\textup{Sel}(G_{K(\sqrt{n})},\mathbb{Z}/4\mathbb{Z})$ with $\textup{rk}_4\, \textup{Cl}(K(\sqrt{n}))$, and we can similarly replace $\dim_{\mathbb{F}_2} \textup{Sel}(G_K,\mathbb{Z}/2\mathbb{Z})$ with $\dim_{\mathbb{F}_2}\textup{Cl}(K)[2]$. The result now follows by combining \Cref{thm:mainanalytic} with \Cref{prop:formula_for_sel_n}.
\end{proof}

\begin{remark}
In the case that $K/\mathbb{Q}$ is imaginary, genus theory gives $\dim_{\mathbb{F}_2}\textup{Cl}(K)[2]=\omega(\Delta)-1,$ so that the formula of \Cref{thm:combined_main} simplifies to show that, for $100\%$ of squarefree $n$, we have 
\[\textup{rk}_4\, \textup{Cl}(K(\sqrt{n}))=\omega_{\textup{inert}}(\Delta_n)+2\omega(\Delta)-3. \]
\end{remark}

\begin{remark} \label{error_term_discussion}
Let us temporarily denote by $f(n)$ the quantity on the right hand side of \eqref{eq:main_theorem_explicit_formula}.  In principle, for $X>0$, the proof of \Cref{thm:combined_main} allows one to obtain an explicit bound for the number
\begin{equation} \label{error_term_formula}
\left|\{n \textup{ squarefree} : |n| \leq X, \textup{rk}_4\, \textup{Cl}(K(\sqrt{n})) \neq f(n)\}\right|
\end{equation}
of exceptional $n$. Contributions to \eqref{error_term_formula} arise both from the use of the results of \Cref{sec:turans_trick} (entering via \Cref{lCores0,lCores2}, and \Cref{thm:main_algebraic_thm}), and the bounds on the average size of $X_n$ and $Y_n$ provided by \Cref{thm:mainanalytic}. As in the statement of that theorem, the latter contributes $O(X \log(X)^{-1/8})$ to \eqref{error_term_formula}. The former has a more serious effect on the quality of the error term with a contribution of size $O(X \log \log(X)^{-1})$. With some extra work we expect that one can obtain an error term of size $O(X \log(X)^{-c})$ for some absolute constant $c > 0$.
\end{remark}

An easy consequence of \Cref{thm:combined_main} is the following result which  determines the distribution of the $4$-rank of $\textup{Cl}(K(\sqrt{n}))$ as $n$ varies. Let $A(X) = \log \log X/2$ and $B(X) = \sqrt{A(X)}$.

\begin{corollary} \label{cor:erdos_kac}
Fix a quadratic extension $K/\Q$. We have for all real $z$
\[
\lim_{X \rightarrow \infty} \frac{|\{n : |n| \leq X, \ n \textup{ squarefree}, \textup{rk}_4\, \textup{Cl}(K(\sqrt{n}))  - A(X) < z B(X)\}|}{|\{n : |n| \leq X, \ n \textup{ squarefree}\}|} = \Phi(z).
\]
\end{corollary}

\begin{proof}
 Define
\[
F(z, X) := \frac{|\{n : |n| \leq X, n \textup{ squarefree}, \ \omega_{\text{inert}}(\Delta_n) - C_K - A(X) < z B(X)\}|}{|\{n : |n| \leq X, n \textup{ squarefree}\}|},
\]
where $C_K$ is a constant depending only on $K$. By Theorem \ref{tMain} it suffices to show that $\lim_{X \rightarrow \infty} F(z, X) = \Phi(z)$. The Chebotarev density theorem implies that
\[
A'(X) := \sum_{p \leq X} \frac{\omega_{\text{inert}}(p)}{p} = A(X) + O(1), \quad B'(X) := \sum_{p \leq X} \frac{\omega_{\text{inert}}(p)^2}{p} = B(X)^2 + O(1).
\]
Then \cite[Theorem 1.3]{KLO} shows that
\[
\lim_{X \rightarrow \infty} \frac{|\{n : |n| \leq X, n \textup{ squarefree}, \ \omega_{\text{inert}}(n) - A'(X) < z \sqrt{B'(X)}\}|}{|\{n : |n| \leq X, n \textup{ squarefree}\}|} = \Phi(z).
\]
But for all $\epsilon > 0$, there is $M > 0$ such that for all $X \geq M$
\begin{align*}
\omega_{\text{inert}}(\Delta_n) - C_K - A(X) < z B(X) &\Rightarrow \omega_{\text{inert}}(n) - A'(X) < (z + \epsilon) \sqrt{B'(X)} \\
\omega_{\text{inert}}(n) - A'(X) < (z - \epsilon) \sqrt{B'(X)} &\Rightarrow  \omega_{\text{inert}}(\Delta_n) - C_K - A(X) < z B(X).
\end{align*}
Hence for all $\epsilon > 0$
\[
\limsup_{X \rightarrow \infty} F(z, X) \leq \Phi(z + \epsilon), \quad \liminf_{X \rightarrow \infty} F(z, X) \geq \Phi(z - \epsilon),
\]
which readily implies the corollary.
\end{proof}

\end{document}